\documentclass{article}
\usepackage{amsmath,amsthm,amssymb,amscd}
\usepackage{tikz}
\usepackage[all,cmtip,color]{xy}
\usepackage[german,english]{babel}
\usepackage{amsmath}
\usepackage{amssymb}
\usepackage{hyperref}
\usepackage{mathrsfs} 
\usepackage{amssymb}
\usepackage{todonotes}
\usepackage{amssymb}
\usepackage[all,cmtip]{xy}
\usepackage{hyperref}
\usepackage[margin = 2.5 cm]{geometry}
\usepackage{colonequals,enumerate}
\usepackage{comment}
\usepackage{bbm}

\usepackage{authblk}
\usepackage{ dsfont }

\usepackage[export]{adjustbox}
\usepackage{graphicx}

\usepackage{enumitem}

\setlist[enumerate,1]{label={(\roman*)}}
\setlist[enumerate,2]{label={(\alph*)}}

\DeclareMathOperator{\Rb}{\mathbb{R}}
\DeclareMathOperator{\Cb}{\mathbb{C}}
\DeclareMathOperator{\Fb}{\mathbb{F}}
\DeclareMathOperator{\Xc}{\mathcal{X}}

\DeclareMathOperator{\Oc}{\mathcal{O}}

\newcommand{\Cc}{\mathcal{C}}
\newcommand{\Gb}{\mathbb{G}}

\newcommand{\Dc}{\mathcal{D}}

\DeclareMathOperator{\N}{\mathsf{N}}

\DeclareMathOperator{\Hom}{\mathsf{Hom}}

\DeclareMathOperator{\St}{\mathsf{st}}

\DeclareMathOperator{\sm}{sm}

\DeclareMathOperator{\IC}{IC}
\DeclareMathOperator{\CO}{\mathcal{O}}

\DeclareMathOperator{\inv}{\mathsf{inv}}

\DeclareMathOperator{\A}{{\mathsf{A}}}

\DeclareMathOperator{\colim}{\mathsf{colim}}
\renewcommand{\lim}{\mathsf{lim}}

\newcommand{\C}{\mathsf{C}}

\newcommand\varto[1]{\mathrel{\hbox to #1pt{\rightarrowfill}}}

\DeclareMathOperator{\Mod}{\mathsf{Mod}}

\DeclareMathOperator{\Nc}{\mathsf{N}}

\DeclareMathOperator{\Sb}{\mathbb{S}}

\DeclareMathOperator{\Lie}{Lie}

\DeclareMathOperator{\Ab}{\mathbb{A}}

\DeclareMathOperator{\Zb}{\mathbb{Z}}

\DeclareMathOperator{\Xb}{\mathbb{X}}
\DeclareMathOperator{\Mb}{\mathbb{M}}

\DeclareMathOperator{\Gal}{Gal}
\DeclareMathOperator{\Aut}{\mathsf{Aut}}
\DeclareMathOperator{\Map}{\mathsf{Map}}
\DeclareMathOperator{\G}{\mathbb{G}}

\DeclareMathOperator{\GL}{GL}
\DeclareMathOperator{\PGL}{PGL}

\DeclareMathOperator{\Sp}{Sp}
\DeclareMathOperator{\M}{{\mathsf{M}}}
\DeclareMathOperator{\Hs}{{\mathsf{H}}}

\DeclareMathOperator{\X}{\mathcal{X}}

\newcommand{\Mc}{\mathcal{M}}

\DeclareMathOperator{\Tr}{\mathsf{Tr}}
\DeclareMathOperator{\Fr}{Fr}

\DeclareMathOperator{\vol}{\mathsf{vol}}
\DeclareMathOperator{\Spec}{\mathsf{Spec}}

\DeclareMathOperator{\Sym}{Sym}

\DeclareMathOperator{\Oo}{\mathcal{O}}

\DeclareMathOperator{\rank}{rk}

\DeclareMathOperator{\Yc}{\mathcal{Y}}

\DeclareMathOperator{\Qb}{\mathbb{Q}}

\newcommand{\BA}{{\mathbb{A}}}

\newcommand{\BC}{{\mathbb{C}}}
\newcommand{\BD}{{\mathbb{D}}}

\newcommand{\BF}{{\mathbb{F}}}
\newcommand{\BG}{{\mathbb{G}}}

\newcommand{\BL}{{\mathbb{L}}}
\newcommand{\BM}{{\mathbb{M}}}
\newcommand{\BN}{{\mathbb{N}}}

\newcommand{\BQ}{{\mathbb{Q}}}

\newcommand{\BX}{{\mathbb{X}}}

\newcommand{\BZ}{{\mathbb{Z}}}

\newcommand{\Bw}{\mathbbm{w}}

\newcommand{\CB}{{\mathcal B}}
\newcommand{\CC}{{\mathcal C}}

\newcommand{\CF}{{\mathcal F}}
\newcommand{\CG}{{\mathcal G}}

\newcommand{\CM}{{\mathcal M}}

\newcommand{\CT}{{\mathcal T}}

\DeclareMathOperator{\Res}{\mathsf{Res}}

\DeclareMathOperator{\Pic}{\mathsf{Pic}}
\DeclareMathOperator{\BPic}{\mathbb{P}\mathsf{ic}}

\newcommand{\FA}{\A}

\newcommand{\BPS}{\mathcal{BPS} }

\DeclareMathOperator{\Ext}{\mathsf{Ext}}

\DeclareMathOperator{\Stab}{\mathsf{Stab}}

\DeclareMathOperator{\Rep}{\mathsf{Rep}}

\newcommand{\defeq}{\colonequals}
\let\into\hookrightarrow

\theoremstyle{definition}
\newtheorem{definition}{Definition}[section]
\newtheorem{construction}[definition]{Construction}
\newtheorem{rmk}[definition]{Remark}

\theoremstyle{plain}
\newtheorem{theorem}[definition]{Theorem}
\newtheorem{proposition}[definition]{Proposition}
\newtheorem{corollary}[definition]{Corollary}

\newtheorem{lemma}[definition]{Lemma}

\newtheorem{example}[definition]{Example}

\begin{document}

\author[1]{Michael Groechenig\thanks{\url{michael.groechenig@utoronto.ca}}}
\author[2]{Dimitri Wyss\thanks{\url{dimitri.wyss@epfl.ch}}}
\author[3]{Paul Ziegler\thanks{\url{paul.ziegler@mathematik.uni-regensburg.de}}}
\affil[1]{Department of Mathematics, University of Toronto}
\affil[2]{\'Ecole Polytechnique F\'ed\'erale de Lausanne}
\affil[3]{Fakult\"at f\"ur Mathematik, Universit\"at Regensburg}

\title{$\chi$-independence for K3-surfaces via $p$-adic integration\let\thefootnote\relax\footnotetext{M.G. was supported by an NSERC discovery grant and an Alfred P. Sloan
fellowship. D.W. was supported by the Swiss National Science Foundation [No. 218340]. P.Z. was supported by the German Research Foundation (Project Number 547483711).}}

\renewcommand\Authands{ and }
\maketitle 

\begin{abstract}
This article provides a proof of a previously unknown case of Toda's $\chi$-independence conjecture by reduction to non-archimedean local fields. Our strategy is based on a novel comparison of Frobenius-traces for BPS sheaves on moduli spaces of objects in 2-Calabi-Yau categories and the integral of the complex-exponentiated Hasse invariant of the obstruction gerbe. This result applies to many cases of interest, including Nakajima quiver varieties, moduli of Higgs bundles and moduli of sheaves on K3 surfaces. Along the way, we describe the local structure of these moduli stacks and spaces over a base of large mixed characteristic.
\end{abstract}

\tableofcontents

\section{Introduction}

The purpose of this note is to prove Toda's $\chi$-independence conjecture \cite[Conjecture 1.2]{To17} for the BPS cohomology of moduli spaces of 1-dimensional sheaves on smooth projective surfaces $X$ with trivial canonical bundle, that is either a K3 or abelian surface.

Let $\beta$ be an ample, base-point free curve class on $X$ and $\chi \in \BZ$. We write $\BM_{\beta,\chi}$ for the moduli stack of pure $1$-dimensional sheaves on $X$ that are Gieseker-semistable with respect to a fixed polarization, have support of class $\beta$ and Euler-characteristic $\chi$. Using dimensional reduction on the local Calabi-Yau 3-fold $X \times \BA^1$ \cite{To17} defines a perverse sheaf, in fact a mixed Hodge module, $\BPS_{\M_{\beta,\chi}}$ on the good moduli space $\M_{\beta,\chi}$, which encodes the mathematical definition of  Gopakumar-Vafa invariants in the presence of strictly semistable objects. The invariants are essentially given as dimensions of graded pieces in the perverse filtration on the cohomology $H^*(\M_{\beta,\chi},\BPS_{\M_{\beta,\chi}})$ induced by the proper Hilber-Chow morphism $ \M_{\beta,\chi} \to B=\mathbb{P} H^0(X,\CO_X(\beta))$. Our first results is the following: 

\begin{theorem}[$\chi$-independence for the perverse filtration]\label{thm:main-2} Let $\beta$ be an ample and base-point free curve class and $\chi,\chi'\in \Zb$. Then there exists an isomorphism of graded vector spaces
\[H^*(\M_{\beta,\chi},\BPS_{\M_{\beta,\chi}}) \cong H^*(\M_{\beta,\chi'},\BPS_{\M_{\beta,\chi'}}),  \]
respecting the perverse filtration. 
\end{theorem}

We can also consider $H^*(\M_{\beta,\chi},\BPS_{\M_{\beta,\chi}})$ as a Hodge structure in which case we obtain:

\begin{theorem}[$\chi$-independence for Hodge numbers]\label{thm:main}
 Let $\beta$ be an ample and base-point free curve class and $\chi,\chi'\in \Zb$. Then, the BPS Hodge numbers satisfy for all $p,q\geq 0$ the following relation
\[
h^{p,q}(\M_{\beta,\chi},\BPS_{\M_{\beta,\chi}}) = h^{p,q}(\M_{\beta,\chi'},\BPS_{\M_{\beta,\chi'}}).
\]
\end{theorem}

In fact, we prove relative versions of these results, see Theorem \ref{chiind}. These were also obtained by Davison--Hennecart--Kinjo--Schiffmann--Vasserot simultaneously by means of a different argument \cite{DHSSV26}.

\subsubsection*{Overview of previous results}

Even though Toda's $\chi$-independence conjecture is expected to hold for a much wider class of moduli problems, lately, the K3 case has received particular attention. The Euler characteristic version of Theorem \ref{thm:main} was proven by Maulik--Thomas in \cite[Theorem 1.2]{maulik2019sheaf} and a K-theoretic refinement follows from \cite{Th24}.

In \cite{KK21}, the $\chi$-independence conjecture was established for the moduli problem of Higgs bundles on a curve $X$. Due to the interpretation of Higgs bundles as coherent sheaves on the cotangent bundle $T^*X$, this case is often described as the \emph{local curve case}. 

\begin{theorem}[Kinjo--Koseki]\label{thm:kinjo-koseki}
Let $\BM_{Dol}(C;r,d)$ denote the moduli stack of semistable Higgs bundles of rank $r$ and degree $d$ on a smooth proper curve $C$ and $\M_{Dol}(C;r,d)$ its good moduli space. Then, for fixed $r$ and arbitrary integers $d,d'$ there exists an isomorphism
\[H^*(\M_{Dol}(C;r,d),\BPS_{\M_{Dol}(C;r,d)}) \cong H^*(\M_{Dol}(C;r,d'),\BPS_{\M_{Dol}(C;r,d')}).\]
\end{theorem}
When $d$ and $d'$ are coprime to $r$, the resulting moduli spaces are smooth and BPS cohomology equals singular cohomology. In the coprime situation, this degree-independence result was conjectured by Hausel \cite[Conjecture 3.2]{Hausel:zr} and independently proven by Mellit in \cite{Me20} and the authors in \cite{GWZ20a}.

Beyond the case of varieties with trivial canonical bundles, the Fano case has received attention. Cohomological $\chi$-independence for moduli of sheaves on a del Pezzo surface with one-dimensional support, was established by Maulik--Shen in \cite{MS20}. In this case, the BPS-cohomology agrees with the classical intersection cohomology of the moduli space $\M$ and the stack $\BM$ happens to be smooth. Both these properties fail in the case of varieties with trivial canonical bundles.

\subsubsection*{Methods}

The strategy is ultimately based on an application of $p$-adic integration in the spirit of Batyrev's paper \cite{batyrev1999birational} and the previous work of the authors. We emphasise that even though this appears to be the fourth paper in a series of articles, it can be read independently of its predecessors.

As a first step, one considers a spreading-out of the moduli stacks $\BM_{\beta,\chi}$ to a base $S$ of mixed characteristic, from which we obtain by a further base change algebraic stacks over the ring of integers $\Oc_F$ of a non-archimedean local field of characteristic $0$ (i.e., a finite extension $F/\Qb_p$ of the field of $p$-adic numbers).

The resulting moduli space $\M_{\beta,\chi}/\Oc_F$ is often a singular $\Oc_F$-scheme. However, its singularities are sufficiently moderate to define a canonical measure $|\omega_{\M_{\beta,\chi}}|$ on the set $\M_{\beta,\chi}(\Oc_F)$ of $\Oc_F$-rational points exploiting the symplectic structure on the smooth locus of $\M_{\beta,\chi}$. For points that belong to the smooth locus of $\M_{\beta,\chi}(\Oc_F)$, integrating with respect to this measure amounts to counting $\Fb_q$-points in the following precise sense:
$$\vol((\M_{\beta,\chi})^{\sm}(\Oc_F)) = \frac{|(\M_{\beta,\chi})^{\sm}(k)|}{q^{\dim \M_{\beta,\chi}}},$$
where $k=\Fb_q$ denotes the finite field defined to be the residue field of the local ring $\Oc_F$.

Near singular points, ramifications arise, and we are unable to offer a similar interpretation of the $p$-adic volume in complete generality. However, the situation improves considerably, once the volume is replaced by the following integral (which only differs from the volume near the singular locus):
$$\vol_{\alpha}(\M_{\beta,\chi}(\Oc_F)) = \int_{\M_{r,\chi}(\Oc_F)} e^{2\pi i \alpha} |\omega_{\M_{\beta,\chi}}|,$$
where $\alpha$ denotes the Brauer class of the obstruction gerbe (to the existence of a universal sheaf) defined over the stable locus of $\M_{\beta,\chi}$. The Brauer group of $F$ is isomorphic to $\Qb/\Zb$ by means of the Hasse invariant. This allows us to define the above exponential for all $\Oc_F$-points $x\in \M_{\beta,\chi}(\Oc_F)$ with the corresponding $F$-point belonging to the stable locus. Since the complement of this set is a zero set, this is sufficient to define $\vol_{\alpha}(\M_{\beta,\chi}(\Oc_F))$.
\begin{theorem}\label{thm:main-3}
For $\BM_{\beta,\chi}/\Oc_F$ obtained by reduction from a complex moduli of objects in a CY2 category, we have whenever the residue characteristic $p$ of $\Oc_F$ is sufficiently large that \[\vol_{\alpha}(\M_{\beta,\chi}(\Oc_F)) =  \frac{\Tr(\varphi|H^*(\BPS_{\M_{\beta,\chi}}))}{q^{\dim \M_{\beta,\chi}}}= q^{-\dim \M_{\beta,\chi}}\sum_{i=0}^{2\dim M_{\beta,\chi}} (-1)^i \Tr\big(\varphi|H^i(\M_{\beta,\chi},\BPS_{\BM_{\beta,\chi}})\big).\]
\end{theorem}

We refer the reader to Theorem \ref{thm:key} for a more precise and general version of the statement above, a proof is given in Subsection \ref{sub:vol-BPS}. It applies to arbitrary moduli stacks and spaces of objects in a CY2-category. In particular, it can be applied to Nakajima quiver varieties and moduli of rank $n$ and degree $0$ Higgs bundles on a curve $C$ and in fact these cases play an important role in the proof of this result.

The first step of the proof is based on rewriting the global formula above as a local statement, from which one reobtains the global assertion as a consequence of the Grothendieck-Lefschetz formula:
\begin{equation}\label{eqn:main-3}\int_{\M_{\beta,\chi}(\Oc_F)_{x}} e^{2\pi i \alpha}|\omega_{\M_{\beta,\chi}}| = \frac{\Tr(\varphi_{x}|(\BPS_{\M_{\beta,\chi}})_{x})}{q^{\dim \M_{\beta,\chi}}}.\end{equation}
In the formula above, $\M_{\beta,\chi}(\Oc_F)_{x}$ denotes the domain of $\Oc_F$-rational points of the moduli space, which specialise to a given $k$-point $x \in \M_{\beta,\chi}(k)$. On the right-hand side, we denote by $\BPS_{\M_{\beta,\chi}}$ the BPS sheaf, or more correctly, an object of the derived category $D^b_{const}(\M_{\beta,\chi},\overline{\BQ}_{\ell})$, to which we apply the Frobenius-trace, as familiar from the \emph{function-sheaf correspondence}. Reminiscent of the behaviour of these moduli of objects in CY2 categories over the complex numbers, one expects these moduli to be modeled on Nakajima quiver varieties, for large residue characteristic $p$. In particular, it suffices to prove Equation \eqref{eqn:main-3} for quiver varieties (and certain arithmetic twists, to be introduced in the following). Nevertheless, these integrals turn out to be exceedingly hard to evaluate, and a natural recursive approach to the proof of Equation \eqref{eqn:main-3} terminates in the case of relatively simple quivers like a $g$-loop quiver.

At this point, the proof of the local statement \eqref{eqn:main-3} reverts to a \emph{global} strategy, where the $\alpha$-corrected volumes for Nakajima quiver varieties and moduli of Higgs bundles interact with each other. The very quivers which lead to directly intractable $p$-adic integrals, also provide a local model for moduli spaces of Higgs bundles near singularities and global information about $p$-adic volumes for these spaces can be brought into play.
The $\chi$-independence result for Higgs bundles (see Theorem \ref{thm:kinjo-koseki}), proven by Kinjo--Koseki \cite{KK21} provides a key ingredient, and ultimately our proof of the K3 case should thus be framed as a reduction to the Higgs case.

Once this connection between BPS cohomology and the gerbe-corrected $p$-adic volume is established, we proceed by exploiting the $\chi$-independence of $\vol_{\alpha}$, which was already observed in \cite{COW21}, using an argument inspired by the authors' first article \cite{GWZ20a} that applied $p$-adic integration to establish the Hausel--Thaddeus conjecture and, simultaneously with Mellit's \cite{Me20}, proved degree-independence of the Betti numbers of the moduli space of Higgs bundles of rank $r$ and degree $d$ whenever $(r,d)=1$.

\begin{rmk}
In the precursor to this paper \cite{GWZ24} we proved an analogue of Theorem \ref{thm:main-3} for moduli of objects in hereditary abelian categories with symmetric Euler form (e.g., the stack of representations of a symmetric quiver, or moduli of semistable vector bundles on a curve). In this case, BPS cohomology of $\BM$ can be identified with intersection cohomology $IH^*(M)$ of the adequate moduli space. A direct proof of an analogous result of Theorem \ref{thm:main-3} was given in \emph{loc. cit.}, based on an explicit evaluation of the gerbe-corrected $p$-adic volume and a computation verifying that a numerical analogue of the Cohomological Integrality relation is satisfied. This relation was then applied to reprove the $\chi$-independence for Fano surfaces, which is the aforementioned result by Maulik--Shen \cite{MS20}.
While the precursor of this paper convinced the authors of the general feasibility to apply the method of $p$-adic integration to problems of interest in enumerative geometry, we emphasise that the results, ideas and methods of proof presented in the present article are independent of \emph{loc. cit.}
\end{rmk}

   The article is organised as follows: In Section \ref{linmod} we establish several results regarding the the local structure of moduli stacks, in particular with a symplectic structure, in mixed characteristics. This is applied to moduli of 2-Calabi-Yau categories in Section \ref{2CYsect}, where we also introduce BPS sheaves. Section \ref{bpsintsect} establishes the connection between BPS-invariants and $p$-adic integrals, Theorem \ref{thm:key},  and finally Section \ref{proofsect} contains the proof of the $\chi$-independence statements.

\medskip

\noindent\textit{Acknowledgements.} The authors believe that the main idea underlying this project would have never occurred to them without the pleasant working conditions provided by TUM and the group of Claudia Scheimbauer during a visit of the authors during which the key ideas of the argument presented here were developed. The gigantic slides in the building of mathematics and informatics deserve particular mention. 
We are grateful to Ben Davison,  Lucien Hennecart, Olivier Schiffmann, Tasuki Kinjo and \'Eric Vasserot for insightful discussions and coordination with the appearance of \cite{DHSSV26}. Furthermore we warmly thank Francesca Carocci,  Davesh Maulik, Tudor P\u{a}durariu, Sebastian Schlegel-Mejia and Tanguy Vernet for interesting conversations. Finally, it is a pleasure to thank the participants and organisers of the SwissMap conference \emph{Representations, Moduli and Duality} and the BIRS conference \emph{Geometry and arithmetic of algebraic stacks} for their helpful comments and encouragement after attending talks by the authors on the subject of this article.
\section{Linearisation of moduli}\label{linmod}

This section collects several results on the local structure of moduli stacks of objects in abelian and derived categories (satisfying additional assumptions). The essence of all results contained therein is well-known to experts. We record generalisations of these results in mixed characteristic, which are essential to the arithmetic treatment given in this paper.

\subsection{Preliminaries on quiver varieties and their twists}

Let $R$ be a ring. In this paper we will use categories of quiver representations in locally free $R$-modules, and certain \'etale twists of these categories. The formal development of these twists is best performed within a slightly modified framework for quivers introduced below. Similar notions abound in the literature and we lay no claims to originality here.

\begin{definition}
An \emph{quiver enriched in $R$-modules $Q_R$} (or short, an $R$-enriched quiver) is a pair $(Q_{0,R},Q_{1,R}=\oplus_{(x,y)\in Q_0^2}E_{xy})$ consisting of a set $Q_0$ and an $R$-module $Q_{1,R}$ graded by the set $Q_0^2$. We say that $Q_R$ is free (or locally free) if all $R$-modules $E_{xy}$ are free (or locally free).
\end{definition}

An example of an $R$-enriched quiver is depicted below in Figure \ref{fig:enter-label}. 
\begin{figure}[h]
    \centering
    \includegraphics[width=0.5\linewidth]{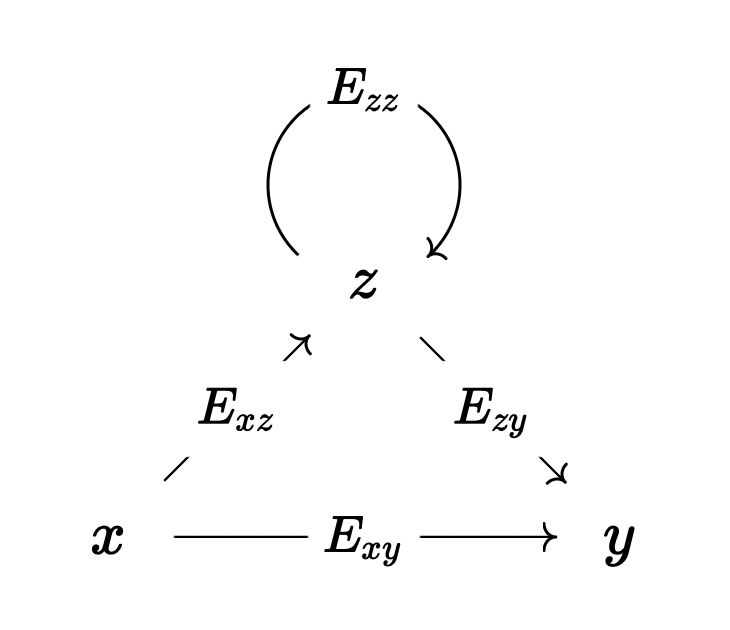}
    \caption{An $R$-enriched quiver with three vertices.}
    \label{fig:enter-label}
\end{figure}
To each quiver $Q$ in the classical sense of a directed graph, we may associate an $R$-enriched quiver $(Q)_R$ as follows. We define $Q_{0,R} = Q_0$ and for every pair of vertices $x,y$ we let $E_{xy}$ be the free $R$-module generated by the set of arrows from $x$ to $y$. The $R$-enriched quivers arising by means of this procedure are precisely those where the $R$-modules $E_{xy}$ are free. The set of directed edges $\vec{E}(x,y)$ from $x$ to $y$ is then in bijection with a basis for $E_{xy}$.

In the following, when describing graphs and quivers, we will adapt lax notation which identifies the isomorphism class of a finite set $\{1,\dots,n\}$ with the integer $n$. In this vein, we will describe a graph $G$ as a pair $(V,E)$, where $V$ is a set and $E\colon V \times V \to \BN$ is a symmetric function. 

For a quiver, or directed graph, we use the related notation $(V,\vec{E})$ where $\vec{E}\colon V \times V \to \BN$ is a function, without any symmetry constraints imposed.

\begin{definition}\label{defi:graph}
\begin{enumerate}
    \item[(a)] To a quiver enriched in locally free $R$-modules $Q_R$ we assign a quiver (a directed graph) by retaining the same set of vertices $Q_0$ and for every pair of vertices $x,y$ setting $\vec{E}(x,y) = \rank_R(E_{xy})$.
    \item[(b)] To a directed graph, we assign a graph by identifying pairs of edges with opposite orientation, yielding
    $$E(x,y)= E(y,x) = \max( \vec{E}(x,y) , \vec{E}(y,x))$$
    for $x\neq y$ and $E(x,x) = \lceil\frac{\vec{E}(x,x)}{2}\rceil$ for the number of loops at every vertex.
    \item[(c)] By composing these two steps, we associate a graph $G=G(Q_R)$ with a quiver $Q_R$ enriched in locally free $R$-modules. We will call $G$ the \emph{portrait} of $Q_R$.
\end{enumerate}
\end{definition}

\begin{definition}
A representation of an $R$-enriched quiver $Q_R$ consists of an $R$-module $M = \oplus_{x \in Q_0}M_x$ graded by $Q_0$, and for every pair of vertices $x,y \in Q_0$ a linear map
\[
E_{xy} \otimes M_x \to M_y.
\]
\end{definition}
It is clear from the above that for a quiver $Q$, there is a canonical equivalence of categories $\Rep_R(Q) \cong \Rep_R((Q)_R)$.

\begin{definition}
A morphism of $R$-enriched quivers $P_R \to Q_R$ is given by a map of sets $P_{0,R} \to Q_{0,R}$ and an $R$-linear morphism $P_{1,R} \to Q_{1,R}$ compatible with the gradings by $P_{0,R}$ respectively $Q_{0,R}$. We denote the set of morphism from $P_R$ to $Q_R$ by $\Hom(P_R,Q_R)$. Likewise, the group of automorphisms of an $R$-enriched quiver will be denoted by $\Aut(Q_R)$.
\end{definition}

\begin{rmk}
This viewpoint allows for a concrete description of the automorphism group $\Aut(Q_R)$ of an $R$-enriched quiver. It is to be understood as the subgroup of $\Aut_R(E)$ of $R$-linear automorphisms of the total edge module, satisfying the condition that for every pair $(i,j)\in I^2$ the direct factor $E_{ij}$ is isomorphically mapped onto another direct factor $E_{i',j'}$ for $(i'j')\in I^2$.
\end{rmk}

The notion of enriched quivers and their representations is only a tautological reformulation of the classical viewpoint on quivers. However, it turns out to be a convenient concept to describe the naturally arising twists of the quiver moduli. In the following, a Galois extension of rings simply denotes a regular \'etale extension $R/R'$ such that $\Spec R' \to \Spec R$ is a torsor under a finite abstract group.

\begin{definition}
Let $R'/R$ be a Galois extension of rings with Galois group $\Gamma$. A $\Gamma$-twisted $R$-enriched quiver is given by an $R'$-enriched quiver $Q_{R'}$ together with an element 
$\tau \in H^1_{\Gal}(\Gamma,\Aut(Q_{R'}))$.
\end{definition}

Not every enriched quiver can be twisted in a non-trivial manner, as non-triviality of the twist is closely linked to a non-trivial permutation of vertices.

\begin{lemma}[No twists for multiple Jordan quivers]\label{lemma:no-twist}
Let $R'/R$ be a Galois-extension of a \emph{local ring} $R$ and $Q_{R'}$ be an $R'$-enriched quiver with $|Q_0| = 1$. Then, we have $H^1_{\Gal}(\Gamma,\Aut(Q_{R'})) = 0$.    
\end{lemma}
\begin{proof}
Since $Q_{R'}$ has a single vertex $x$ with $E_{xx} = E'$ as $R'$-module of arrows, we have $\Aut(Q_{R'}) = \GL_{R'}(E')$. It then follows from Hilbert 90 that the degree $1$ Galois cohomology group is Zariski-locally zero, as asserted above. In particular, for a local ring $R$, this Galois cohomology group vanishes.    
\end{proof}

\begin{rmk}
The lemma above serves as justification for introducing the notion of enriched quivers. The symmetry group of an $g$-Jordan quiver is the symmetric group $S_g$. The Galois cohomology group $H^1_{\text{\'et}}(R,S_g)$ is often non-zero, for instance this is the case for $R = \mathbb{F}_q$. As the above result shows, embedding $S_g$ into the linear group, leads to annihilation of this Galois cohomology class.      
\end{rmk}

\begin{definition} In the following we denote by $Q_R$ an $R$-enriched quiver and by $(Q_{R'},\tau)$ a pair consisting of an $R'$-enriched quiver and a Galois cocycle $\tau \in H^1(\Gamma,\Aut(Q_{R'}))$.
\begin{enumerate}
    \item[(a)] We denote by $RQ_R$ the $R$-linear path algebra of $R$, which is defined to be the $R$-algebra generated by the set $\{e_x \ |\ x \in Q_0\}$ and the $R$-module $E_R=\bigoplus_{(x,y)\in Q_0^2} E_{xy}$ subject to the following relations
    $$(e_x)^2 = e_x,\text{ } e_x\cdot{}v = v\text{ for }v \in E_{xy},\text{ } v \cdot{} e_y = v\text{ for }v\in E_{yx}\text{ and }w\cdot{}v = 0\text{ if }v\in E_{xy}\text{ and }w\in E_{zz'}\text{ for }y\neq z.$$ 
    \item[(b)] We denote by $RQ_{R'}^{\tau}$ the twisted path algebra. It is defined as the form of the path algebra $R'Q_{R'}$ given by the cocycle induced by the map $H^1(\Gamma,\Aut(Q_{R'})) \to H^1(\Gamma,\Aut(R'Q_{R'}))$. 
    \item[(c)] The $R$-linear abelian category $\Rep^{\tau}(Q_{R'})$ is defined to be the category of $RQ^{\tau}_{R'}$-modules.
\end{enumerate}    
\end{definition}

In the following, we will also denote the twisted path algebra by the shorthand $RQ^{\tau}$, and thereby suppress the $R'$-enrichment from notation. In this vein, we will also suppress the appearance of $R$ and $R'$ in the linear quivers in other constructions, whenever there is little risk that this abuse of notation might cause confusion.

\begin{definition}[Twists of moduli of quiver representations]\label{defi:twisted-rep}
Given a triple $(R'/R, Q_{R'}, \tau)$ consisting of a Galois extension $R'/R$, an $R'$-enriched quiver, and a twist $\tau \in H^1(\Gamma,\Aut(Q_{R'})$, we define $\BM_{Q_R,d}^{\tau}$ to be the stack, which assigns to an $R$-algebra $S$ the groupoid of $(RQ_{R'}^{\tau} \otimes_R S)$-representations on locally free $S$-modules with dimension vector $d$. 
\end{definition}

Another perspective on twisted representations is provided by the following explicit description of the stack of twisted representations $\BM^{\tau}_{Q_{R'}}$ as a \emph{twist} of the stack $\BM_{Q_R}$.

It is well-known from the classical viewpoint on representations of quivers that the stack $\BM_{Q_R,d}$ of $Q$-reprentations of a fixed dimension vector $d$ can be described as an explicit quotient stack
$$\BM_{Q_R,d} = [\Ab_{Q_R, d}/\Gb_{d}],$$
where $\Gb_{d} = \prod_{i \in Q_0} \GL_{d_i}$ and $\Ab_{Q_R,d}$ denotes the affine space
$$\prod_{(x,y)\in Q_0^2}\Spec \Sym \Hom(R^{d_y},R^{d_x}\otimes E_{xy}).$$

\begin{construction}\label{const:twisted-rep-stack}
For $R'/R$ as above, and $Q_{R'}$ an $R'$-enriched quiver, we obtain a natural Galois action of $\Gamma=\Gal(R'/R)$ on $\Gb_{Q_{R'},d}$ and $\Ab_{Q_{R'},d'}$, which is compactible with the action
$$\Gb_{d} \times \Ab_{Q_{R'},d} \to \Ab_{Q_{R'},d}.$$
It therefore makes sense to twist both the group scheme and the action by a Galois cocycle $\tau \in H^1(\Gamma,\Aut(Q_{R'}))$. We will denote the resulting quasi-split reductive group scheme over $R$ by $G^{\tau}_{Q,d}$ and the resulting affine space obtained by twist by $\Ab_{Q,d}^{\tau}$. One then has
\begin{equation}\label{eqn:mom-map}
\BM^{\tau}_{Q_{R},d} \cong [\Ab_{Q,d}^{\tau}/G^{\tau}_{Q,d}].
\end{equation}
\end{construction}

Given a quiver $Q$, the Euclidean space $\Rb^{Q_{0}}/\Rb = \Stab(Q)$ is referred to as the space of stability conditions for $Q$. This space comes with a wall and chamber decomposition, which should be thought of as a partition into convex equivalence classes $[\theta]$. For $\theta,\theta' \in \Stab(Q)$ satisfying $[\theta]=[\theta']$ we obtain that the resulting sub-categories of semistable objects
$$\Rep^{\theta-ss}(Q) = \Rep^{\theta'-ss}(Q)$$
agree.

If $(Q_{R'},\tau)$ is a linear enrichment of the quiver $Q$ together with a twist $\tau$ with respect to a Galois extension $R'/R$,  we may consider the resulting action of $\Gal(R'/R)$ on the set of vertices $Q_0$. This action induces an action of the same group on $\Stab(Q)$. For any $\theta$ with $[\theta]$ being a fixed point for the action, we may define
$$\Rep_{\theta-ss}^{\tau}(Q_R) \subset \Rep^{\tau}(Q_{R'}).$$
For any dimension vector $d \in (\Zb^{Q_0})^{\Gal(R'/R)}$ the moduli stack of objects in this $R$-linear category is denoted by $\BM_{Q,d}^{\theta,\tau}$.

\begin{lemma}
The $R$-stack $\BM_{Q,d}^{\theta,\tau}$ possesses an adequate moduli space $\M_{Q_R,d}^{\theta,\tau}$. There is an isomorphism $(\M_{Q,d}^{\tau})_{R'} \simeq \M_{Q_{R'},d}$.
\end{lemma}
\begin{proof}
The asserted existence holds \'etale-locally by construction as a twist. The statement over the base $R$ follows from the universal property of adequate moduli spaces and a descent argument.
\end{proof}

In this paper, we will only consider twisted quiver representations where $\Gal(R'/R)$ is a finite cyclic group. In this case, the twist $\tau$ essentially corresponds to the choice of an element $\phi$ of a certain order in the group $\Aut(Q)$ of diagram automorphisms of the quiver. Furthermore, it will be natural to consider the following class of diagram automorphisms $\sigma \in \Aut(Q)$.
\begin{definition}\label{defi:impartial}
Let $G = (V,E)$ be a graph and $\sigma \in \Aut(G)$ be a graph automorphism. We say that $G$ is $\sigma$-\emph{impartial} if for every pair of distinct vertices $i,j \in V$ such that $i\neq \sigma(j)$ we also have the equation 
$$ E(i,j) = E(i,\sigma(j)) ,$$
where $E(i,j)=E(j,i)$ denotes the set of undirected edges $i \to j$ in $Q$ and we also require that the equation
$$E(i,i) =  E(i,\sigma^{\ell}(i))  -1 =  E(\sigma^{\ell}(i),\sigma^{\ell}(i))  $$
holds whenever $\sigma^{\ell}(i) \neq i$. 
\end{definition}

Impartiality of $\Gamma$ with respect to $\sigma$ implies that for two distinct vertices $i \neq j$ the number of edges connecting them only depends on the $\langle\sigma\rangle$-orbit of $i$ and $j$. Furthermore, it stipulates the multiplicity of loops at a vertex $i$ to be $1$ less than the multiplicity of edges connecting it to other vertices in the $\langle \sigma \rangle$-orbit. The latter relation reflects the fact that the Euler characteristic is  locally constant in flat families. We refer the reader to the end of the proof of Lemma \ref{lemma:sigma-impartial} for the geometric context.

\begin{example}
We give several examples and non-examples.
\begin{enumerate}
    \item Let $G$ be the graph given by a square with positive orientation. Let $\sigma$ be the rotation by angle $\pi$. Then, $G$ is $\sigma$-impartial. However, if $\sigma$ is chosen to be the rotation by $\frac{\pi}{2}$ then $\sigma$-impartiality no longer holds. This can be fixed, by replacing the square by a complete graph on four vertices.   
    \item Let $G$ be a complete graph on $n$-vertices and $\sigma$ an arbitrary permutation of the vertices. Then, $G$ is $\sigma$-impartial.
    \item Let $G$ be the Dynkin diagram $A_5$ with appropriately chosen orientations (e.g., all edges are oriented to point towards the centre of the graph). Let $\sigma$ be a non-identity automorphism given by a reflection in a line passing perpendicularly through the centre. Then, $G$ is not $\sigma$-impartial.
\end{enumerate}
\end{example}

The notion of $\sigma$-impartiality is related to a desirable property for twisted enriched quivers.

\begin{definition}\label{defi:sigma-impartial-linear}
Let $R'/R$ be a cyclic Galois  extension with $\sigma$ a generator of its Galois group. A pair $(Q_{R'},\tau)$ will then be called $\sigma$-impartial, if the induced automorphism on the portrait $G$ of $Q_{R'}$ (see Definition \ref{defi:graph}) is $\sigma$-impartial.  
\end{definition}

\subsection{Twists of Nakajima quiver varieties}

The construction of Nakajima quiver varieties can also be performed within the framework of $R$-enriched quivers and their twists. The reader familiar with the classical theory will notice that we do not include framings at vertices. Framed quiver varieties will not play a role in this paper. Thus, we construct quiver varieties either by following the algebraic process
\[
\text{Quiver $Q$} \rightsquigarrow \text{Double quiver $DQ$} \rightsquigarrow \text{Preprojective algebra }\Pi_Q^0 = kDQ/(\text{preproj. relations})
\]
or, by defining the Nakajima quiver variety $N_{DQ}$ attached to a quiver $Q$ and a dimension vector $d$ to be the GIT quotient
$$\mu^{-1}_{Q,d}(0)//\Gb_{d},$$
where $$\mu\colon T^*\Ab_{Q,d} \to \mathfrak{g}_{d}^{\vee} = \Lie(\Gb_d)^{\vee}$$ 
is the momentum map\footnote{We remind the reader that the commonly used word \emph{moment map} appears to be a mistranslation of the French term \emph{application moment}.} for the $\Gb_d$-action on $T^*\Ab_{Q,d}$. This GIT quotient also arises as the adequate moduli space of the Artin stack given by the cotangent stack $T^*\BM_{Q,d}$.

Two distinct quiver structures $Q_1$ and $Q_2$ on the same set of vertices may give rise to the same double quiver $DQ_1 = DQ_2 = DQ$. The resulting quiver variety $\N_{DQ}$  and stack $\BN_{DQ}$ only depend on the double quiver $DQ$ and are insensitive to the choice of $Q_1$ over $Q_2$. 

For the case of quivers enriched in locally free modules (and their twists), it is therefore natural to suppress reference to $Q$ altogether, when defining Nakajima quiver varieties. The starting point is the linear analogue of a double quiver, which we will refer to as \emph{symplectic quiver}.

\begin{definition}
Let $DQ_R=(DQ_{0},DQ_{1,R})$ be a quiver enriched in locally free $R$-modules. A \emph{symplectic structure} $\omega$ on $Q_R$ is given by a graded alternating $2$-form on the $Q_0$-graded edge module $Q_{1,R} = \bigoplus_{(x,y) \in Q_0^2} E_{xy}$, such that for every distinct pair $x\neq y$, $\omega$ yields an $R$-linear isomorphism
$$E_{xy} \simeq E_{yx}^{\vee},$$
and for every $x \in Q_0$, $\omega$ restricts to an $R$-linear non-degenerate form on $E_{xx}$. We denote by $\Sp_{\omega}(Q_R) \subset \Aut(Q_R)$ the subgroup of linear automorphisms preserving the symplectic form $\omega$.
\end{definition}

Every symplectic quiver enriched in locally free $R$-modules $(DQ_R,\omega)$ possesses a Lagrangian subquiver $Q_R$, which we simply understand to be a graded Lagrangian subspace of the edge module 
$$L \subset Q_{1,R}=\bigoplus_{(x,y)\in Q_{0}^2} E_{xy}.$$
The symplectic form then yields a canonical isomorphism $Q_{1,R} = L \oplus L^{\vee}$. Such a Lagrangian $L$ can be viewed as the edge module of a linear quiver $Q_R$, whose double quiver (in the enriched sense) is $DQ_R$.

For a triple $(DQ_{R'},\omega,\tau)$, where $\tau \in H^1(\Gamma,\Sp_{\omega}(DQ_{R'}))$, the associated Nakajima quiver stack and variety can be twisted by the cocycle $\tau$, for the simple reason that they are evidently $\Aut(Q_{R'})$-equivariant. More precisely, they are defined in the following manner. 

\begin{definition}
For a triple $(DQ_{R'},\omega,\tau)$, where $\tau \in H^1(\Gamma,\Sp_{\omega}(DQ_{R'}))$ we define the twisted Nakajima quiver stack $\BN^{\tau}_{DQ_R'}$ to be the $R$-stack given by the closed substack
$$[\mu^{\tau,-1}_{Q_{R'},d}(0)/G_d^{\tau}] \subset [\Ab_{DQ,d}^{\tau}/G_{d}^{\tau}],$$
where $\mu^{\tau}$ is the momentum map from \eqref{eqn:mom-map}.
We refer to the GIT quotient $\mu^{\tau,-1}_{Q_{R'},d}(0)//G_d^{\tau}$ as the twisted Nakajima quiver variety $\N_{DQ,d}^{\tau}$.
\end{definition}

By virtue of construction, $\N_{DQ,d}^{\tau}$ is a form of the affine $R$-scheme $\N_{DQ,d}$ and thereby itself an affine $R$-scheme.

\begin{rmk}
    It would also be possible to define a twisted preprojective algebra $\Pi_{Q}^{\tau}$, as a quotient of the $R$-linear (twisted) path algebra of an appropriately defined double quiver and define the stack $\BN_{DQ}^{\tau}$ as the stack of $\Pi^{\tau}_Q$-representations.
\end{rmk}

\begin{definition}
For every $\theta \in \Stab_Q$, we denote the resulting open substack of $\theta$-semistable representations of the preprojective algebra and dimension vector $d$ by $\BM^{\theta}_{Q,d}$. If $d$ and $\theta$ are fixed by the $\Gamma$-action, then we obtain a well-defined open substack
$\BN_{DQ,d}^{\theta,\tau} \subset \BN_{DQ,d}^{\tau}.$
\end{definition}



\subsection{Semisimplicity of representations of reductive groups} \label{SemisimplicitySS}
Let $\bar k$ be an algebraically closed field of characteristic $p>0$ and $G$ a (connected) reductive group over $\bar k$. 
\begin{definition}
    Let $T$ be a maximal torus of $G$ and $B$ a Borel subgroup of $G$ containing $T$. This determines a set of roots $\Phi$ of $G$ as well as a subset $\Phi^+ \subset \Phi$ of positive roots. For $\alpha \in \Phi$ we denote by $\alpha^\vee$ the coresponding coroot.
    \begin{enumerate}
        \item For any character $\lambda \colon T \to \Gb_m$ let
        \begin{equation*}
            n_G(\lambda) \defeq \sum_{\alpha \in \Phi^+}\langle \lambda,\alpha^\vee \rangle.
        \end{equation*}
        \item If $V$ is a finite-dimensional representation of $G$ over $\bar k$, then we let $n_G(V)$ be the maximum of the integers $n_G(\lambda)$ over all weights $\lambda$ of $T$ appearing in $V$.
    \end{enumerate}
\end{definition}

\begin{example} \label{NComp}
    Let $(V,\psi)$ be a non-zero symplectic space. Then the integer $n_{\Sp_\psi}(V)$ associated to the tautological representation $V$ of the associated symplectic group $\Sp_\psi$ is given by $n_{\Sp_\psi}(V)=\dim(V)-1.$

    This can be computed as follows: We may assume that $\psi$ is the standard symplectic form on $k^{2N}$ for some $N\geq 1$ defined by the matrix
    \begin{equation*}
        J=
        \begin{pmatrix}
            0 & 1_N \\
            -1_N & 0
        \end{pmatrix}.
    \end{equation*}
    We consider the maximal torus $T$ of $\Sp_\psi \subset \GL_{2N}$ consisting of diagonal matrices with invertible entries $(t_1,\hdots,t_N,t_1^{-1},\hdots,t_N^{-1})$ as well as the maximal torus $B$ given by the stabilizer of the maximal isotropic subspace $k^N \oplus 0 \subset k^{2N}$. 
    
    Then, if for $1\leq i \leq N$ we let $e_i \colon T \to \Gb_m$ be the projection onto the $i$-th diagonal factor, the resulting positive roots are given by $e_i - e_j$ for all $i<j$ and $-e_i-e_j$ for all $i\leq j$ with coroots $(e_i - e_j)^\vee=e_i^* - e_j^*$ and $(-e_i-e_j)^\vee=-e_i^*-e_j^*$ for $i<j$ and $(2e_i)^\vee=e_i^*$.

    So for any weight $\lambda=\sum_{s=1}^N \lambda_s e_s\colon T \to \Gb_m$ we find
    \begin{equation*}
        n_{\Sp_\psi}(\lambda)=\sum_{i<j}-2\lambda_j+\sum_i \lambda_i=\sum_{i=1}^N (-2i+1)\lambda_i.
    \end{equation*}

    The weights of $T$ in the tautological representation $V$ are the homomorphisms $\pm e_i$ for all $1\leq i \leq N$. So we find
    \begin{equation*}
        n_{\Sp_\psi}(V)=n_{\Sp_\psi}(-e_N)=2N-1.
    \end{equation*}
\end{example}
\begin{definition}
    A subgroup $\Gamma \subset G(\bar k)$ is $G$-cr if for any parabolic subgroup $P$ of $G$ satisfying $\Gamma \subset P(\bar k)$, the group $\Gamma$ is already contained in $L(\bar k)$ for some Levi subgroup $L$ of $P$.
\end{definition}

\begin{theorem}[{Eugene, c.f. \cite[Theorem 7]{SerreMoursund}}] \label{EugeneThm}
    Let $H \subset G$ be a reductive subgroup and $p \geq h_G$, where $h_G$ is the Coxeter number of $G$. Then $H(\bar k)$ is $G$-cr.
\end{theorem}

\begin{theorem}[{Serre, c.f. \cite[Theorem 6]{SerreMoursund}}] \label{SerreThm}
    Let $V$ be a representation of $G$ with $n_G(V) < p$. Then for any $\Gamma \subset G(\bar k)$ which is $G$-cr, the resulting $\Gamma$-action on $V$ is semisimple.
\end{theorem}

Let now $k \subset \bar k$ be a finite subfield with algebraic closure $\bar k$ and assume that $G$ is defined over $k$. We note the following:
\begin{lemma} \label{SSLemma}
    \begin{enumerate}
        \item Let $V$ be a representation of $G$ over $k$ whose base change to $\bar k$ is semisimple. Then $V$ is semisimple over $k$.
        \item If two semisimple $G$-representations $V$ and $V'$ over $k$ become isomorphic after base change to $\bar k$, then they are isomorphic over $k$.
    \end{enumerate}
    
\end{lemma}
\begin{proof}
    (i) Let $\sigma \in \Gal(\bar k / k)$ be the Frobenius generator. Then there exists a decomposition $V_{\bar k} =\oplus_i V_i$ into simple subrepresentations $V_i$ which are permuted by $\sigma$. This can be constructed inductively: Start with any simple factor $V_1$ of $V_{\bar k}$. If $\sigma(V_1)=V_1$, then $V_1$ is defined over $k$ and we may split off $V_1$ over $k$ and continue with the quotient $V/V_1$. Otherwise we let $V_2=\sigma(V_1)$ and continue according to whether $\sigma(V_2) \subset V_1 \oplus V_2$ or not, etc.

    Then the direct summands of $V_{\bar k}$ given by the sums of $\sigma$-orbits of the $V_i$ give a decomposition of $V$ over $k$ into simple representations.

    (ii) The sheaf of $G$-equivariant isomorphisms $V \cong V'$ is a torsor under the sheaf of $G$-equivariant automorphisms of say $V$. By the semisimplicity of $V$, the base change of this sheaf to $\bar k$ is representable by a product of general linear groups. Hence by flat descent, this group sheaf is representable by a connected affine group scheme over $k$. Using Lang's theorem this implies the claim.
\end{proof}

\subsection{Torsors under centralizers}
The goal of this subsection is to prove Theorem \ref{CentralizerThm} below.

We fix a symplectic space $(V,\psi)$ over a finite field $k$ and denote the associated symplectic group scheme by $\Sp_\psi$. 
\begin{definition}
    Consider an indecomposable symplectic representation $G \to \Sp_\psi$ from a reductive group scheme $G$.
    \begin{enumerate}
        \item Such a representation is of \emph{type 1} if it is irreducible as a $G$-module.
        \item Such a representation is of \emph{type 2} if it is of the form $V=U \oplus U^*$ for some irreducible $G$-representation $U$ with the symplectic form given by
        \begin{equation*}
            \omega(u_1+u_1^*,u_2+u_2^*)=\langle u_1^*,u_2\rangle - \langle u_2^*,u_1 \rangle.
        \end{equation*}
    \end{enumerate}
\end{definition}

We will need the following, which is essentially due to Knop.
\begin{theorem} \label{KnopThm}
    Consider two symplectic representations $\iota, \iota'\colon G \to \Sp_\psi$ from a reductive group scheme $G$ which are semisimple as $G$-representations.
\begin{enumerate}
    \item If $\iota$ is indecomposable as a symplectic representation, then it is either of type 1 or 2.
    \item In general, every such symplectic representation $\iota$ is a direct sum of summands of type 1 or 2. These summands are unique up to isomorphism and permutation.
    \item If $\iota$ and $\iota'$ are isomorphic as $G$-representations and at least one of them does not possess a type 1 summand, then they are $\Sp_\psi(k)$-conjugate.
\end{enumerate}
\end{theorem}
\begin{proof}
    This follows by the argument of \cite[Thm 2.1]{Knop}.
\end{proof}

\begin{theorem} \label{CentralizerThm}
   Let $k$ be any finite field of characteristic $p \geq \dim(V)$ and $\iota\colon G \to \Sp_{\psi}$ any injective homomorphism from a reductive group $G$. If this symplectic representation does not contain any symplectic subrepresentations of type 1, then the centralizer $Z=Z_{\Sp_{\psi}}(G)$ has trivial cohomology set $H^1(k,Z)$.
\end{theorem}
\begin{proof}
 Let $\CT$ be a $Z$-torsor over $k$. Since $\Sp_\psi$ is special, the pushout of $\CT$ to $\Sp_{\psi}$ is trivial, i.e. we may embed $\CT$ equivariantly into $\Sp_{\psi}.$ By conjugating $\iota$ with any section of $\CT(\bar k)$ we obtain an injective homomorphism $\iota'\colon G \to \Sp_{\psi}$ which is independent of the choice of section and hence is defined over $k$. Then $\CT \subset \Sp_{\psi}$ is the $Z$-torsor of sections conjugating $\iota$ to $\iota'$, and so we want to show that $\iota$ and $\iota'$ are $\Sp_\psi(k)$-conjugate.

    By Example \ref{NComp}, the integer $n_{\Sp_\psi}(V)$ is equal to $\dim(V)-1$. As explained in Lecture 2 of \cite{SerreMoursund}, the fact that the representation $V$ of $\Sp_{\psi}$ is faithful implies $n_{\Sp_\psi}(V)\geq h_{\Sp_{\psi}}-1$. Hence by Theorem \ref{EugeneThm}, the group $G(\bar k)$, considered as a subgroup of $\Sp_{\psi}(\bar k)$ via either $\iota$ or $\iota'$ is $\Sp_{\psi}$-cr. We consider $V$ as symplectic representation of $G$ via $\iota$ or $\iota'$. Then Theorem \ref{SerreThm} implies that these are geometrically semisimple as $G$-representations, and Lemma \ref{SSLemma} implies that they are semisimple as $G$-representations. By the construction of $\iota'$, these $G$-representations are geometrically isomorphic, and so they are isomorphic by Lemma \ref{SSLemma} (ii). So by Theorem \ref{KnopThm} they are $\Sp_\psi(k)$-conjugate.
\end{proof}

We will also need smoothness of symplectic centralizer group schemes $Z$ of type $2$ representations, over a more general base.

\begin{lemma}\label{lemma:Z-smooth}
Let $G_R / R$ be a connected reductive group over a complete dvr $R$ with a residue field $k$ of characteristic $p$. Let $(V,\psi)$ be a symplectic $G$-representation on a free $R$-module satisfying the inequality $\dim_R(V) \geq p$. If the special fibre $V_k$ has no symplectic subrepresentation of type $1$, then the centralizer $Z=Z_{\Sp_{\psi}}(G)$ is a smooth $R$-scheme.
\end{lemma}
\begin{proof}
To establish smoothness, using faithfully flat descent we may assume without loss of generality that $k$ is algebraically closed. Then, as in the proof of Theorem \ref{CentralizerThm}, the $G_k$-module $V_k$ is semisimple.

Let $(V_k,\psi) \simeq \bigoplus_{i=1}^m (\overline{U}_i^{m_i} \oplus (\overline{U}_i^{*})^{m_i})$ be a type $1$ decomposition into indecomposable symplectic subrepresentations of the module $V_k = V \otimes_R k$ with $U_i \neq U_j$ for $i\neq j$. Since the obstruction to lifting each summand $\overline{U}_i$ modulo a square-zero ideal lies in $\Ext^1(\overline{U}_i,V_k/\overline{U}_i)=0$, we see that the $k$-linear subrepresentations $\overline{U}_i$ can be lifted to $R$-linear and free subrepresentations $U_i \subset V$, which generate $V$ by Nakayama's lemma. 

Denoting the fraction field of $R$ by $K$, it is clear that each $U_i \otimes_R K$ is $K$-linearly indecomposable. Indeed, a $G_K$-invariant subspace of $U_i \otimes_R K$ extends by the valuation criterion applied to the Grassmannian of $U_i$, to a free $R$-linear subspace of $U_i$, which is still $G_R$-invariant.

Therefore, we have 
\begin{equation}\label{eqn:dec}
V = \bigoplus_{i=1}^m (U_i^{m_i} \oplus (U_i^{*})^{m_i}).
\end{equation}

The centralizer group scheme $Z_{\Sp_{\psi}}(G )$ agrees with the group scheme of symplectic automorphisms $\Aut_G(V,\psi)$ of $(V,\psi)$. Since the residue field $k$ is assumed to be algebraically closed, we may use the $R$-linear decomposition of Equation \eqref{eqn:dec} to identify it with the product
$$Z_{\Sp_{\psi}}(G) = \prod_{i=1}^m Z_i,$$
where $Z_i = \Sp_{2m_i}$ if $U_i \simeq U_i^*$ and $Z_i = \GL_{m_i}$ otherwise.
\end{proof}

\subsection{On the existence of symplectic linearisations}

We denote by $\Xc$ an algebraic $R$-stack, where we assume $R$ to be a Henselian dvr with finite residue field $k$ of characteristic $p$ of Krull dimension $\leq 1$ (e.g., $
R=k$ or $R$ equals the ring of integers $\Oc_F$ of a local field $F$). 
Furthermore, we assume that the schematic locus of $\Xc$ is dense and its complement has codimension $\geq 2$. 

\begin{definition}[Weakly symplectic stacks]
A weak symplectic form $\omega$ on $\Xc$ is given by an equivalence $\BL_{\Xc/R}^{\bullet} \simeq (\BL_{\Xc/R}^{\bullet})^{\vee}$, where $\BL^{\bullet}_{\Xc/R}$ denotes the cotangent complex, such that the induced pairing on $\mathcal{H}^0(\BL^{\bullet}_{\Xc/R})$ is alternating. 
\end{definition}

It is clear that a $0$-shifted symplectic stack over $\Cb$ in the sense of \cite{pantev2013shifted} gives rise to a weak symplectic structure as above. The main difference is that we will not concern ourselves with specifying what it means for $\omega$ to be a closed $2$-form.

\begin{definition}\label{defi:well-pointed}
Let $\bar{x} \in \Xc(k)$ be a point with reductive stabiliser group $G_k$. We say that $(\Xc,\bar{x})$ is \emph{well-pointed} if $\bar{x}$ can be lifted to an $R$-point $x$, such that $\Aut_{\Xc}(x) = G_R$ is a reductive group $R$-scheme with $k$-fibre $G_k$.    
\end{definition}

This notion of well-pointedness acts as a useful simplifying assumption. We will see in Lemma \ref{lemma:well-pointed} that for moduli problems of interest, it is always satisfied.

We denote by $S/R$ an \'etale extension. In this section we will investigate descent for certain results on the local structure of $\Xc$, adapted to the symplectic nature of $\Xc$.

Let $\Yc_1$ and $\Yc_2$ denote $R$-stacks, endowed with $k$-points $\bar{y}_i \in \Yc_i(k)$. We refer to a pair $(\Yc_i,\bar{y}_i)$ as a pointed $R$-stack and to the fibre product
$$\Hom_{\bullet}\big((\Yc_2,\bar{y}_2),(\Yc_1,\bar{y}_1)\big)=\Hom(\Yc_2,\Yc_1)\times_{\Yc_1(k)}\{\bar{y}_1\}$$
as the groupoid of pointed morphisms, or alternatively, the groupoid of morphisms respecting the base point. Note that by virtue of definition of fibre products in the $(2,1)$-category of groupoids, this is equivalent to the groupoid of pairs $(f,i)$ consisting of a morphism $f\in \Yc_1(\Yc_2)$ and an isomorphism $i\colon f(\bar{y}_2) \simeq \bar{y}_1$ relating the base points.

\begin{definition}\label{defi:Sigma}
Assume that $(\Xc,\bar{x})$ is well-pointed and we fix a lift of $\bar{x}$ to $x \in \Xc(R)$ as in Definition \ref{defi:well-pointed}. We denote by $T^*V$ the symplectic affine space given by $\mathcal{H}^0((\BL^{\bullet}_{\Xc/R})^{\vee}_x)$ with its natural $G$-action. Let $\mu\colon T^*V \to \mathfrak{g}^{\vee}$ be the corresponding momentum map. Consider the following stack on the small \'etale site of $R$  
$$\Sigma_{\Xc,\bar{x}}(S)=\{\iota\colon(\widehat{(\Xc_S)}_{\bar{x}},\bar{x})\dashrightarrow ([\widehat{(\mu^{-1}(0))_S} / G_S],0)\subset [\widehat{(T^*V)_S} / G_S] | \iota^*\omega = \omega\text{ and }\iota\colon \Aut_{\Xc_S}(\bar{x}) \simeq G_{S_k}\text{ is defined over }k \}.$$ 
\end{definition}
We remark that completion in $0$ of the affine space $T^*V$ or $\mu^{-1}(0)$ is omitted from the notation in the definition above.

We first record an observation pertaining to the local structure of $\Xc$ near $\bar{x}$.
\begin{lemma}\label{lemma:lci}
If there exists an \'etale extension $S/R$ such that $\Sigma_{\Xc,\bar{x}}(S) \neq \emptyset$, then $\Xc$ is a locally complete intersection in a Zariski-open neighbourhood of $\bar{x}$.
\end{lemma}
\begin{proof}
The property of being a locally complete intersection is Zariski-open and is invariant under completions. 
The assertion then follows directly from the fact that $\mu^{-1}(0)_S$ is a locally complete intersection, since it is a subset of codimension $\dim \mathfrak{g}$ defined as the joint vanishing locus of $\dim \mathfrak{g}$ regular functions.
\end{proof}

The following lemma states that the formal completions used in the definition above, can be replaced by an algebraic counterpart. For this purpose, for a pointed scheme $(Y,\bar{x})$ we denote by $\widehat{Y}_{\bar{x}}^{\mathrm{alg}}$ the affine scheme $\Spec \widehat{\Oc}_{Y,\bar{x}}$ of the completion in $\bar{x}$.
\begin{definition}
$\Sigma^{\mathrm{alg}}_{\Xc,\bar{x}} = \{\upsilon\colon ([\widehat{(\mu^{-1}(0))^{\mathrm{alg}}_S} / G],0)\dashrightarrow ((\Xc_S),\bar{x}) | \upsilon^*\omega = \omega \text{ and }\upsilon^{-1} \colon \Aut_{\Xc_S}(\bar{x}) \simeq G_{S,}\text{ is defined over }k \}$ 
\end{definition}
\begin{lemma}\label{lemma:Sigma-alg}
$\Sigma_{\Xc,\bar{x}}=\Sigma^{\mathrm{alg}}_{\Xc,\bar{x}}$.
\end{lemma}
\begin{proof}
This follows from the following string of equivalences
$$\Hom_{\bullet}(((\widehat{\Yc}_S)_{\bar{y}},\bar{y}),((\widehat{\Xc}_S)_{\bar{x}},\bar{x})) \simeq \Hom_{\bullet}\big((\widehat{\Yc}_S)_{\bar{y}},\bar{y}),(\Xc_S,\bar{x})\big) \simeq\Hom_{\bullet}\big((\widehat{\Yc}_S^{\mathrm{alg}},\bar{y}),(\Xc_S,\bar{x})\big)$$
holding for algebraic stacks $\Xc_S$ and $\Yc_S = [\mu^{-1}(0)_S/G_S]$ locally of finite presentation, where we denote by $(\widehat{\Yc}^{\mathrm{alg}}_S)_{\bar{y}}$ the formal quotient stack $[\widehat{\mu^{-1}(0)}^{\mathrm{alg}}_S/G_S]$.
\end{proof}

\begin{lemma}
The groupoid $\Sigma_{\Xc,{\bar{x}}}(S)$ is equivalent to a set.
\end{lemma}
\begin{proof}
 Since the base $S$ is irrelevant to these considerations, reference to $S$ will be omitted for the duration of the proof. It suffices to show that an isomorphism of stacks
$\widehat{\Xc}_{\bar{x}} \to [\widehat{\mu^{-1}(0)}_0/G]$
does not have any non-trivial automorphisms if it exists. We remark that some of the defining properties of $\Sigma_{\Xc,\bar{x}}$, such as pointedness and symplecticity will not play a role in the argument. Assuming that a morphisms as above exists, we may formally identify $\widehat{\Xc}_{\bar{x}}$ and $[\widehat{\mu^{-1}(0)}/G]$.

The morphism $[\widehat{\mu^{-1}(0)}/G] \to [\mu^{-1}(0)/G]$ can be algebraised as stated in Lemma \ref{lemma:Sigma-alg}. We therefore obtain a diagram
$$[\widehat{\mu^{-1}(0)}^{\mathrm{alg}}/G ]\leftarrow P \xrightarrow{f} \mu^{-1}(0),$$
where $P$ is a $G$-torsor and $f$ is a $G$-equivariant morphism. We assert that the automorphism group of the $G$-torsor $P$ acts on the collection of such diagrams with trivial stabiliser group.  This is the case since the stabiliser group on $\Xc$ is assumed to be trivial up to codimension $2$, i.e., there exists an open dense subset $\Xc^s \subset \Xc$, which is a space. The same property must then hold for $[\mu^{-1}(0)/G]$ by virtue of assumption. Furthermore, since $\mu^{-1}(0)$ is a locally complete intersection by Lemma \ref{lemma:lci}, so is the $G$-torsor $P$. In particular, the latter is Cohen-Macaulay. We therefore see that a morphism $f\colon P \to \mu^{-1}(0)$ can be recovered from its restriction $f|_{P^s}$. We therefore see that $\Aut P$ acts freely on the set of $G$-equivariant morphisms $\{f\colon P \to \mu^{-1}(0)\}^G$.
\end{proof}
Before turning to the next definition, we will explain the strategy used in the previous proof with the help of an example. In the case of interest to this paper, the symplectic quotient $[\mu^{-1}(0)/G]$ corresponds to the Nakajima quiver stack $\Nc_Q$, that is, the moduli stack of representations of the preprojective algebra $\Pi_Q^0$.

In this particular case, the argument above amounts to the observation that a morphism $\widehat{\Xc}^{\mathrm{alg}} \to \Nc_Q$ yields a family of $\alpha$-twisted $\Pi^0_Q$-representations. Using generic stability, one easily sees that such a twisted representation has only scalar automorphisms, which are erased by the $\G_m$-rigidification process.

\begin{definition}\label{defi:Sb}
Let us denote by $\mathbb{S}$ the group of symplectic automorphisms of $[\widehat{\mu^{-1}(0)}/G]$  preserving the base point $0$. That is, as a set-valued sheaf  $\mathbb{S}=\Sigma_{[\mu^{-1}(0)/G],0}$, which we endow with the canonical group structure given by composition. 
\end{definition}
\begin{lemma}
The natural action
$\mathbb{S}\circlearrowright \Sigma_{\Xc,\bar{x}}$
is a torsor, if $\Sigma_{\Xc,\bar{x}} \neq \emptyset$. 
\end{lemma}
\begin{proof}
This follows directly from Definitions \ref{defi:Sigma} and \ref{defi:Sb}. 
\end{proof}
\begin{lemma}\label{lemma:def-theory}
The group sheaf $\mathbb{S}$ admits an inverse limit description:
$$\mathbb{S} \simeq \varprojlim_{n \geq 1} \mathbb{S}^{(n)},$$
where $\mathbb{S}^{(n)}$ is defined to be the group of pointed symplectic automorphisms of the $n$-th order infinitesimal neighbourhood $[\mu^{-1}(0)^{(n)}/G].$
The fibres of the transition maps in the projective system above are additive groups $\G_a^r$. 
\end{lemma}
\begin{proof}
As a consequence of deformation theory, the kernel of the map
$$\Aut(([\mu^{-1}(0)^{(n+1)}/G],0)) \rightarrow \Aut([\mu^{-1}(0)^{(n)}/G],0)$$
is given by $$\mathbb{H}^0([\mu^{-1}(0)^{(n)}/G],\mathcal{I}\otimes (\mathbb{L}^{\bullet}_{[\mu^{-1}(0)/G]})^{\vee}) \simeq \big(\mathcal{H}^0(\mu^{-1}(0)^{(n)},\mathcal{I}\otimes (\mathbb{L}^{\bullet}_{\mu^{-1}(0)})^{\vee}\big)^G,$$
where $\mathbb{L}^{\bullet}$ denotes the cotangent complex and $\mathcal{I}$ the sheaf of ideals corresponding to the closed immersion $[\mu^{-1}(0)^{(n)}] \hookrightarrow [\mu^{-1}(0)^{(n+1)}/G]$. 
This follows directly from the fact that $\mu^{-1}(0)^{(n)}$ is affine and therefore the kernel of 
$\Aut((\mu^{-1}(0)^{(n+1)},0)) \rightarrow \Aut(\mu^{-1}(0)^{(n)},0)$ is given by $\mathbb{H}^0(\mu^{-1}(0)^{(n)},\mathcal{I}\otimes (\mathbb{L}^{\bullet}_{[\mu^{-1}(0)/G]})^{\vee})$.
\end{proof}
We remark that we do not assert surjectivity of the transition maps $\Sb^{(n+1)} \to \Sb^{(n)}$. However, the map $\Sb \to \Sb^{(1)}$ will be shown to be surjective in the following lemma.

\begin{lemma}\label{lemma:split-surjection}
There exists a split surjection $r\colon \mathbb{S} \to Z = Z_{\Sp(T^*V)}(G)$, with the kernel being a (geometrically) connected pro-unipotent group scheme.
\end{lemma}
\begin{proof}
The surjection $r$ can be explicitly defined as the differential in $0$:
$$r=d_0\colon \mathbb{S} \to \Sp(V^{\vee}\oplus V).$$
By $G$-equivariance, it factors through the centralizer scheme $Z$. Surjectivity of this map follows from the existence of a splitting to which we turn in the next paragraph.

By definition, $Z=\Sp(V^{\vee}\oplus V)^G$, and therefore there is a natural morphism
$$Z \to \Hom(([T^*V/G],0),([T^*V/G],0)).$$
Furthermore, by the $G$-equivariance of the momentum map
$\mu\colon T^*V \to \mathfrak{g}^{\vee}$, the linear symplectic $Z$-action on $T^*V$ preserves $\mu^{-1}(0) \subset T^*V$.
By means of this construction, we obtain a splitting $i\colon Z \to \mathbb{S}$.

The statement on the kernel being a connected pro-algebraic group is a consequence of the deformation-theoretic description of Lemma \ref{lemma:def-theory}. To see this, one notes that $\mathbb{S}^{(1)}$, the group of pointed symplectic automorphisms of the first order infinitesimal neighbourhood of $0$ in $\mu^{(-1)}(0)$, agrees with $Z$. We therefore obtain from \emph{loc. cit.} that the fibre of the surjection
$\mathbb{S} \twoheadrightarrow \mathbb{S}^{(1)}=Z$ is given by a pro-unipotent group scheme, which is geometrically connected.
\end{proof}

\begin{definition}\label{defi:symplectic-linearisation}
 An isomorphism of pointed formal stacks 
$$\iota\colon(\widehat{(\Xc_S)}_{\bar{x}},\bar{x})\dashrightarrow ([\widehat{(\mu^{-1}(0))_S} / G],0)$$ 
 preserving $\omega$ and inducing an isomorphism of group schemes $\Aut_{\Xc}(\bar{x}) \simeq G_k$ is called a \emph{formal symplectic linearisation near $\bar{x}$}.
\end{definition}

\begin{proposition}\label{prop:symp-linear}
   If $\Sigma_{\Xc,\bar{x}} \neq \emptyset$ and $p \geq \dim \mathcal{H}^0((\BL_{\Xc/R})_{\bar{x}}^{\bullet})$ and $(\Xc,\bar x)$ is well-pointed  (see Definition \ref{defi:well-pointed}), then the $\Sb$-torsor $\Sigma_{\Xc,\bar{x}}$ is trivial. Hence, there exists a formal symplectic linearisation near $\bar{x}$ under these assumptions.
\end{proposition}
\begin{proof}
As seen in Lemma \ref{lemma:split-surjection} there is a split surjection $r\colon \Sb \to Z=Z_{Sp(T^*V)(G)}$. The kernel $K$ being pro-unipotent, we have $H^1_{\text{\'et}}(\Spec R,K) =0$. Similarly, $H^1_{\text{\'et}}(\Spec R,Z)=H^1_{\text{\'et}}(\Spec k,Z_k)$ since $Z$ is smooth by Lemma \ref{lemma:Z-smooth}, and $H^1_{\text{\'et}}(\Spec k,Z_k)$ vanishes, by virtue of Theorem \ref{CentralizerThm}, as $p$ satisfies the stated inequality. 
\end{proof}

The result above should be understood as a descent statement for the existence of formal symplectic linearisations. The non-triviality condition on the torsor amounts to the existence of a formal symplectic linearisation near $\bar{x}$, after performing an \'etale base change (and thereby possibly enlarging the residue field in $\bar{x}$). The proposition states that it is then also possible to find a formal symplectic linearisation without passing to an \'etale neighbourhood first. 

For an abelian group $A$ we denote by $A' \subset A$ the union of $A[n]$ where $n$ ranges over all integers satisfying $n \notin p\Zb$.
\begin{lemma}\label{lemma:Br-stack}
\begin{itemize}
    \item[(a)] $H^2_{\rm fppf}(BG,\Gb_m) \simeq \{\text{central extensions of $G$ by $\G_m$}\}/\text{iso}$. 
    \item[(b)] Assume that $\widehat{\Xc}_{\bar{x}}$ possesses a formal symplectic linearisation, that is, it is equivalent to a quotient stack $[\widehat{Y}_x/G]=[\widehat{\mu^{-1}(0)}/G]$, where $G=\Aut_{\Xc}(x)$, where $x$ is the fixed lift of $\bar{x}$ used in Definition \ref{defi:Sigma}. Pullback along the natural morphism $i_{\bar{x}}\colon BG \to \widehat{\Xc}_{\bar{x}}$ induces an isomorphism
$H^2_{\rm fppf}(\widehat{\Xc}_{\bar{x}},\Gb_m)' \simeq H^2_{\rm fppf}(BG,\G_m)'$.
\item[(c)] Any $\G_m$-gerbe $\widehat{\BX}_x \to \widehat{\X}_x$ which is of order $n \notin p\Zb$ is equivalent to $[\widehat{\mu^{-1}(0)}/\widetilde{G}]$, where $\widetilde{G}$ is a central extension of $G$ by $\Gb_m$.
\end{itemize}
\end{lemma}
\begin{proof}
Assertion (a) is a well-known property in abstract group cohomology (see \cite[Theorem 6.6.3]{Weibel}). Adopting a functorial perspective, one directly reduces the algebro-geometric assertion (a) to the classical case. 

Assertion (b) is well-known for Brauer classes on schemes. We include a proof. Denote by $Y^{(n)}_{\bar{x}}$ the $n$-th order infinitesimal neighbourhoods such that $\widehat{Y} = \varinjlim Y_{\bar{x}}^{(n)}$ as a formal scheme. We write $\Xc^{(n)}_{\bar{x}}$ for the quotient stack $[Y^{(n)}_{\bar{x}}/G]$. We denote by $\mathcal{I}_n$ the square-zero sheaf of ideals on $\Xc^{(n+1)}_{\bar{x}}$ corresponding to this closed immersion. Note that this ideal is actually scheme-theoretically supported on $BG = [\Spec k/G]$. In particular, multiplication by $p$ is the zero map on $\mathcal{I}_n$. 

It suffices to show that $H^2_{\rm fppf}(\Xc^{(n+1)}_{\bar{x}},\Gb_m)' \to H^2_{\rm fppf}(\Xc^{(n)}_{\bar{x}},\Gb_m)'$ is an isomorphism. There is a short exact sequence of fppf sheaves
$$0 \to \mathcal{I} \to \Oc_{\Xc^{(n+1)}}^{\times} \to \Oc_{\Xc^{(n)}}^{\times} \to 0.$$

The prime-to-$p$ parts of the higher cohomology groups $H^i_{\rm fppf}(BG,\mathcal{I}_n)' = 0$ appearing in this long exact sequence vanish. To see this, we use that multiplication by $p$ is the zero map on $\mathcal{I}_n$, as mentioned previously.

Given a extension $\widetilde{G}$ of $G$ by $\G_m$, the quotient stack $[\widehat{Y}_{\bar{x}}/\widetilde{G}] \to [\widehat{Y}_x/G] = \widehat{\Xc}_{\bar{x}}$ is a $\G_m$-gerbe $\alpha$ on $\widehat{\Xc}_{\bar{x}}$ such that the class of $i_{\bar{x}}^*\alpha \in H^2_{\rm fppf}(BG,\G_m)$ corresponds to $\widetilde{G}$. This proves part (c).
\end{proof}

As a direct consequence, we obtain the following corollary.

\begin{corollary}\label{cor:symp-lin}
Let $\Xc/R$ be a weakly symplectic stack with adequate moduli space $\pi\colon \Xc \to X$. Let $\BX \to \Xc$ be a $\Gb_m$-gerbe on $\Xc$, which we assume to be of order $n \notin p\Zb$.

Assume that there exists a covering in the \'etale topology $\{U_i \to X\}_{i \in I}$ and a collection of triples $(G_i,\rho_i,\widetilde{G}_i)$ consisting of a reductive group $G_i$, a representation $\rho_i\colon G_i \circlearrowright V_i$ and a central extension $\widetilde{G}_i$ of $G_i$ by $\Gb_m$ such that we have commutative diagrams
\[
\xymatrix{
\BX \times_X U_i \ar[r]^-{\simeq} \ar[d] & [T^*V_i // \widetilde{G}_i] \ar[d] \\
U_i \ar[r]^-{\simeq} & T^*V_i // G_i,
}
\]
where we recall that $[T^*V_i//\widetilde{G}_i]$ denotes the quotient stack $[\mu_i^{-1}(0)/\widetilde{G}_i]$, and $\mu_i\colon T^*V_i \to \mathfrak{g}_i^{\vee} = \Lie(G_i)^{\vee}$ the momentum map for the $G_i$-action on $T^*V_i$. Then, $\Xc$ is formally symplectically linearisable near $\bar{x} \in X(k)$ (see Definition \ref{defi:symplectic-linearisation}) if the assumption $p \geq \dim \mathcal{H}^0((\BL_{\Xc/R})_{\bar{x}}^{\bullet})$ holds and $(\Xc,\bar{x})$ is well-pointed (see Definition \ref{defi:well-pointed}). Furthermore, under the same assumptions on $p$ and $\bar{x}$ there exists an equivalence of formal stacks
$$\widehat{\BX}_{\bar{x}} \simeq [\widehat{\mu^{-1}(0)}/\widetilde{G}_{\bar{x}}],$$
inducing the formal symplectic linearisation after $\Gb_m$-rigidification.
\end{corollary}
\begin{proof}
Let $\bar{x} \in X(k)$ be a $k$-point and $i \in I$ an index such that $\bar{x}$ lies in the image of the \'etale-open $U_i \to X$. Let us choose a finite field extension $k'/k$ and a $k'$-point $\bar{y}$ above $\bar{x}$. 

It follows from the assumption that there is an equivalence $\widehat{(\Xc \times_X U_i)}_{\bar{y}} \simeq [\widehat{T^*V_i}//G_i]$. This presentation as a symplectic quotient stack needs to be modified since (the special fibre of) $G_i$ could be strictly larger than $\Aut_{\Dc}(\bar{x})$.

Let us denote by $\mathfrak{h} \subset \mathfrak{g}_i = \Lie(G_i)_k$ the Lie algebra of $H=\Aut_{\Xc}(\bar{x}) = \Stab_{(G_i)_k}(\bar{x})$. The properties of the momentum map then give us a diagram
$$\mathfrak{g}_i/\mathfrak{h} \hookrightarrow T^*V \twoheadrightarrow (\mathfrak{g}_i/\mathfrak{h})^{\vee}.$$
We conclude that the syplectic representation $T^*V$ contains the symplectic direct summand $\mathfrak{g}_i/\mathfrak{h} \oplus (\mathfrak{g}_i/\mathfrak{h})^{\vee}$. Here, we use the semisimplity property following from our assumption on $p$ using the results recalled in Subsection \ref{SemisimplicitySS}.

As a consequence of the symplectic semisimplicity of $T^*V$, there exists a complementary symplectic subrepresentation $W \subset T^*V$, such that $W \oplus \mathfrak{g}_i/\mathfrak{h} \oplus (\mathfrak{g}_i/\mathfrak{h})^{\vee} \simeq T^*V$. Note that there exists an isomorphism $W \simeq T^*U$ for a representation $U$ since $W$ is also of type 2 as a symplectic representation.

Consequently, we have that the natural map $$[W//H] \to [T^*V//G_i]$$ induces an equivalence of stacks in the formal neighbourhood of the image of $\bar{y}$.

Thus, $\Sigma_{\Xc,\bar{x}}$ is non-empty, thus it is a trivial $\Sb$-torsor by virtue of Proposition \ref{prop:symp-linear}, and there exists a formal symplectic linearisation near $\bar{x}$.
The last assertion follows directly from part (c) of Lemma \ref{lemma:Br-stack}.
\end{proof}

\section{Moduli of objects in 2-Calabi-Yau categories} \label{2CYsect}

\subsection{Definitions}
Let $R$ be a finitely generated $\BZ$-algebra and $S = \Spec(R)$. Consider an $R$-linear abelian category $\Cc$ as in \cite[Section 5.2]{GWZ24}. In particular, we assume that $\Cc$ is extension closed inside a larger auxiliary cocomplete category $\CB$ used to define the ext-groups $\Ext^i(X,Y)$ for objects $X,Y \in \Cc$ and $i\geq 0$. 

Henceforth, we assume that $\Cc$ is $2$-Calabi-Yau in the sense that $\Ext^i(X,Y) = 0$ for all $X,Y \in \Cc$ and $i > 2$ and for $i \in \{0,1,2\}$ there are bilinear pairings
\begin{equation}\label{cy2} \delta^i\colon \Ext^i(X,Y) \times \Ext^{2-i}(Y,X) \to R,\end{equation}
satisfying $\delta^i(A,B) = (-1)^i\delta^{2-i}(B,A)$.
We further require that $\delta^i$ be non-degenerate over every closed point $\Spec(k) \to S$. Finally, we assume that $\Cc_\BC$ is $2$CY as in \cite{Da24}.

 Let $\BM = \BM_\Cc$ denote the algebraic stack of objects in $\Cc$. We denote by $\pi_0(\BM)$ the \'etale sheaf on $S=\Spec R$, assigning to an \'etale-open $S' \to S$ the set of connected components of $\BM \times_S S'$.
 
 We will assume that each component of $\BM$ admits a moduli space in an appropriate sense:

\begin{enumerate}
\item\label{assumption:i} Each component $\BM_\gamma$ of the disjoint union 
\[\BM = \sqcup_{\gamma \in \pi_0(\BM)(S)} \BM_\gamma,\]
is a finite type linear quotient stack over $R$.

\item For each $\gamma \in \pi_0(\BM)(S)$ there exists an adequate moduli space $\tilde{\pi}_\gamma \colon \BM_\gamma \to \M_\gamma$.
\end{enumerate}

We write $\CM_\Cc = \sqcup_{\gamma \in  \pi_0(\BM)(S)} \CM_\gamma$ for the $\BG_m$-rigidification along the scalar automorphisms and $\pi_\gamma\colon \CM_\gamma \to \M_\gamma$ for the induced morphisms. For any closed point $\Spec(k) \to S$ we define the Euler-pairing
\begin{align*}
(\cdot,\cdot)_k \colon \BM(k) \times \BM(k) &\to \BZ,\\
X,Y &\mapsto \sum_{i=0}^2 (-1)^i \dim_k \Ext^i(X,Y).
    \end{align*}

The pairing $(\cdot,\cdot)_k$ locally constant and by \eqref{cy2} symmetric and thus defines a non-degenerate symmetric bilinear form on the monoid $(\pi_0(\BM)(k),\oplus)$. 

In many concrete examples (e.g., for the abelian category of semistable Higgs bundles of a fixed slope $\mu$ on a curve $C_S/S$), it is true (or true up to making appropriate choices) that the \'etale sheaf $\pi_0(\BM)$ is constant. In this case, one has the convenient property $\pi_0(\BM)(S) = \pi_0(\BM)(\Cb)$ at one's disposal.

While we do not know whether every complex-linear CY2 category $\CC_{\Cb}$ possesses a spreading out with this nice property, it is possible to always achieve the following.

\begin{definition}\label{assumption:i}
Let $\Cc_{\Cb}$ be a complex-linear CY2 category and $\gamma \in \pi_0(\BM)(\Cb)$ a connected component of its moduli stack of objects over $\Cb$. A spreading-out $\Cc_R$ to an $R$-linear CY2 category, where $R \subset \Cb$ is a finitely generated subring, is called \emph{$\gamma$-adapted} if for every pair $(\gamma_1,\gamma_2) \in \pi_0(\BM_{\Cc})(\Cb)$ with $\gamma_1\oplus \gamma_2 = \gamma$ we have that the germs of sections $\gamma_1,\gamma_2 \in \pi_0(\BM_{\Cc})_{\Spec \Cb}$ extend to global sections in $\pi_0(\BM_{\Cc})(S)$.    
\end{definition}

\begin{rmk}
The situation of which Definition \ref{assumption:i} attempts to steer clear from is where a connected component $\BM_{\gamma} \subset \BM_{\Cc_S}$ splits into several connected components after base change to $\Cb$. For instance, it could happen that the sheaf of connected components $\pi_0(\BM)$ is \'etale-locally constant and that the resulting action of $\pi_1^{\text{\'et}}(S,\Spec \Cb)$ on $\pi_0(\BM)(\Cb)$ is non-trivial. For a given $\gamma$, we we can always adjoin further elements to $R = \Gamma(S,\Oc_S)$ to ensure that $\gamma$ and all its direct summands $(\gamma_1,\gamma_2)$ are fixed by this action.
\end{rmk}

\begin{lemma}\label{lemma:well-pointed}
Every $\bar{x} \in \M_{\gamma}(k)$ can be extended to an $R$-point of $\Mc_{\gamma}(k)$. In particular, $(\Mc_{\gamma},\bar{x})$ is well-pointed in the sense of Definition \ref{defi:well-pointed}.   
\end{lemma}
\begin{proof}
The $k$-point $\bar{x}$ corresponds to a semisimple object $\bar{x} = \bigoplus_{i=1}^m \bar{x}_i$, where $\bar{x}_i \in \M_{\gamma_i}^s(k_i) = \Mc_{\gamma_i}^s(k_i)$ for an appropriate field extension $k_i$. Since $\Mc_{\gamma_i}^s$ is smooth, it is possible to extend $\bar{x}_i$ to an $R_i$-point $x_i$, where $R_i/R$ denotes the corresponding unramified extension of $R$. The direct sum of these extensions is the requisite $x \in \Mc_{\gamma}(R)$ that extends $\bar{x}$.   
\end{proof}


\subsection{Linearisation in 2-Calabi-Yau categories}\label{2cylin}

We now recall a well-known construction through which linear quivers naturally arise from semisimple objects in abelian categories.   When applying this definition to the abelian categories of semistable sheaves of a fixed slope $\mu$, one recovers the familiar construction of the Ext-quiver associated to a polystable sheaf.

\begin{definition}
Let $\Cc$ be an $R$-linear abelian category and $X=\bigoplus_{i\in I} X_i^{d_i}$ a semisimple object with $I$ being a finite set and each factor $X_i \in \Cc$ being defined over $R$ and $X_i \neq X_j$ for $i\neq j$. The $R$-enriched Ext-quiver is defined to be the pair $(Q_0,E)$, where $Q_0=I$ and $E_{ij} = \Ext(X_i,X_j)$ for all $i,j \in I$. 
\end{definition}
\begin{construction}
For a Galois-extension $R'/R$ as before and an $R$-linear abelian CY2 category $\Cc_{R}$ we may associate a pair $(DQ_{R'},\tau)$ to any semisimple $X=\bigoplus_{i\in I} X_i^{d_i}$ as above, where we assume that $X_i \in \C_{R'}$  which is fixed by $\Gamma$ up to isomorphism. The $R'$-enriched quiver is the $R'$-linear Ext-quiver, as above, note that it is automatically doubled (by the CY2 condition) and the twisting data is obtained from the natural descent datum of $X\otimes_R R'$: by assumption 
$V_R=\Ext^1(X,X)\in \Mod(R)$
is an $R$-module and the base change to $R'$
$$(\Ext^1(X,X))_{R'} \simeq \bigoplus_{(i,j) \in I^2} \Ext^1(X_i,X_j)$$
is endowed with a natural $I^2$-grading. This structure is equivalent to the notion of an $R'$-enriched quiver. The natural Galois-action on this $R'$-module yields a Galois $1$-cocycle taking values within the subgroup $\Aut(DQ_{R'}) \subset \GL(V_{R'})$. The induced element of $H^1(\Gamma,\Aut(DQ_{R'}))$ is the requisiste twist $\tau$.\end{construction}
\begin{lemma}\label{lemma:sigma-impartial}
Let $X = \bigoplus_{i\in I} X_i^{d_i} \in \Cc_R$ be a semisimple object in $\Cc_R$ with simple factors in $\Cc_{R'}$ and let $R \to k$ be a homomorphism to a finite field $k$. Denoting by $\gamma \in \pi_0(\BM)(S)$ the connected component containing $[X]$, we assume that $\Cc_R/R$ is $\gamma$-adapted. 

Then, the twisted quiver obtained from $X\otimes_R k$ is $\sigma$-impartial (see Definition \ref{defi:impartial}), where $\sigma$ denotes the Frobenius automorphism $\sigma \in \Gal(k)$.
\end{lemma}
\begin{proof}
Let us denote $X\otimes_R R' \in \Cc_{R'}$ by abuse of notation by $X$. Since $X \simeq \sigma(X)$, the simple factors $X_i$ can be grouped into $\langle\sigma\rangle$-orbits $X_i^{d_i} \oplus X_{\sigma(i)}^{d_i} \oplus \cdots \oplus X_{\sigma^{\mathrm{ord}(\sigma)-1}(i)} = (\bigoplus_{j=1}^{\mathrm{ord}(\sigma)} X_{\sigma^j(i)})^{d_i}$. This induces an action of $\langle\sigma\rangle$ on the set of simple factors $I$ preserving the dimension vector $d$. We claim that the Ext-quiver satisfies the $\sigma$-impartiality condition from Definition \ref{defi:impartial}.
Let us assume that $i,j$ is a pair of distinct vertices with $X_{\sigma(i)} \neq X_j$. Then, we have that $\Hom(X_i,X_j) = \Ext^2(X_j,X_i) = 0$. Thus, the rank of the module $\Ext^1(X_i,X_j) = E_{ij}$ is equal to $-(X_i,X_j)$. Since the Euler characteristic of $R\Hom(-,-)$ is locally constant in flat families and the connected component $\BM_{\gamma}$ is defined over $k$ by virtue of assumption that $\Cc_R/R$ is $\gamma$-adapted (see Definition \ref{assumption:i}), we obtain $X_{i},X_{\sigma(i)} \in \BM_{\gamma}$ and thus
$\dim \Ext^1(X_i,X_j) = \dim \Ext^1(X_{\sigma(i)},X_j)$. 

Likewise, we obtain that $\dim Hom(X_i,X_i) - \dim Ext^1(X_i,X_i) + \dim Ext^2(X_i,X_i) =2-\dim \Ext^1(X_i,X_i) = -\dim \Ext^1(X_i,X_{\sigma(i)})$, whenever $\sigma(i) \neq i$. Here, we used that $X_i$ is simple and therefore only has scalar endomorphisms, implying that $\dim Hom(X_i,X_i) = 1$. By the CY2 property, the same conclusion holds for $\dim Ext^2(X_i,X_i)$. 
This concludes the proof that the twisted quiver obtained from $X \otimes_R k$ is $\sigma$-impartial.
\end{proof}

\begin{proposition}[Linearisation for CY2 moduli]\label{prop:K3-linear}
Let $\Cc_R /R$ be a $\gamma$-adapted spreading-out of a complex-linear CY2 category $\Cc_{\Cb}$. Then, there exists a finite type morphism $S' = \Spec R' \to S = \Spec R$ with $R \subset R' \subset \Cb$ such that for every closed point $x \in \M_{\gamma,S'}$ there exists a $\sigma$-impartial linear twisted quiver datum $(Q_x,d_x,\tau_x)$ such that we have the following commutative diagram of equivalences for formal stacks and formal schemes:
\[
\xymatrix{
(\widehat{\BM}_{\gamma})_x \ar[r]^-{\simeq} \ar[d] & (\widehat{\BN}_{DQ_x,d_x}^{\tau_x})_0 \ar[d] \\
(\widehat{\M}_\gamma)_x \ar[r]^-{\simeq} & (\widehat{\N}_{DQ_x,d_x}^{\tau_x})_0 
}
\]
\end{proposition}
\begin{proof}
According to Davison's Neighbourhood Theorem \cite[Theorem 5.11]{Da24} there exists a finite cover in the \'etale topology
$$\mathcal{U}=\{U_i \to \M_{\gamma,\Cb}\}$$
of the complex moduli space $\Mc_{\gamma,\Cb}/\Cb$, and pairs $(Q_i,d_i)$ consisting of a quiver and a dimension vector such that we a commutative diagram 
\[
\xymatrix{
(\BM_{\gamma,\Cb}) \times_{\M_{\gamma,\Cb}} U_i \ar[r]^-{\text{\'et}} \ar[d] & \BN_{DQ_i,d_i} \ar[d] \\
U_i \ar[r]^-{\text{\'et}} & \N_{DQ_i,d_i}
}
\]
of stacks. 

We now choose a finitely generated extension $R \subset R' \subset \Cb$ such that $\mathcal{U}$ and the commutative diagrams spread out to $R' \subset \Cb$ and denote $\Spec R'$ by $S'$.

This spreading-out then satisfies the assumptions of Corollary \ref{cor:symp-lin} by virtue of construction, since $\BN_{DQ_i,d_i}$ is defined by symplectic reduction. We infer that $\BM_{R'}$ is formally symplectically linearisable near every closed point $x \in |M_{\gamma,S'}|$.
\end{proof}
\begin{corollary}[Addendum to the linearisation]\label{indlin}
Let $\Cc_R/R$ and $\gamma$ be as in Proposition \ref{prop:K3-linear}. The commutative diagram of equivalences for formal stacks can be extended for any decomposition $x = x_1 \oplus x_2$ in the following way:
\[
\xymatrix{
 (\widehat{\BM}_{\gamma_1})_{x_1} \times (\widehat{\BM}_{\gamma_2})_{x_2} \ar[r]^-{\simeq} \ar[d]_{\oplus} & (\widehat{\BN}^{\tau_1}_{DQ_x,d_{x_1}})_{0} \times (\widehat{\BN}^{\tau_2}_{DQ_x,d_{x_2}})_0\ar[d]^{\oplus} \\
(\widehat{\BM}_{\gamma})_x \ar[r]^-{\simeq} \ar[d] & (\widehat{\BN}_{DQ_x,d_x}^{\tau_x})_0 \ar[d] \\
(\widehat{\M}_\gamma)_x \ar[r]^-{\simeq} & (\widehat{\N}_{DQ_x,d_x}^{\tau_x})_0 
}
\]
\end{corollary}
\begin{proof}
This follows directly from the observation that for a moduli stack of objects $\BM$ in a category $\Cc$, we have a commutative diagram
\begin{equation}\label{diag:oplus}
\xymatrix{
\BM \times \BM \ar[r] \ar[d]_-{\oplus} & I_{\mu_2}\BM \ar[ld] \\
\BM,
}
\end{equation}
where $I_{\mu_2}\BM$ denotes the mapping stack $\Map(B\mu_2,\BM)$.

Applying the functor $I_{\mu_2}(-)$ to the equivalence $(\widehat{\BM})_x \simeq (\widehat{\BN}^{\tau_x}_{DQ_x,d_x})_0$ we obtain a commuting square
\[
\xymatrix{
I_{\mu_2}(\widehat{\BM})_x \ar[r]^-{\simeq} \ar[d] & I_{\mu_2}(\widehat{\BN}^{\tau_x}_{DQ_x,d_x})_0 \ar[d] \\
(\widehat{\BM})_x \ar[r]^-{\simeq} & (\widehat{\BN}^{\tau_x}_{DQ_x,d_x})_0
}
\]
Since both stacks are moduli stacks of objects in a category, the top row of this diagram is naturally isomorphic to $(\widehat{\BM}_{\gamma_1})_{x_1} \times (\widehat{\BM}_{\gamma_2})_{x_2} \simeq (\widehat{\BN}^{\tau_1}_{DQ_x,d_{x_1}})_{0} \times (\widehat{\BN}^{\tau_2}_{DQ_x,d_{x_2}})_0$. This concludes the proof.
\end{proof}

\subsection{BPS sheaves for 2-Calabi-Yau categories}

Over $\BC$ an explicit description of the BPS sheaf of a 2-Calabi-Yau category is given in \cite[Definition 7.15]{DHS23} as a perverse sheaf $\BPS_{\M}$, in fact even a mixed Hodge module, on the good moduli space $\M = \M_{\Cc}$ of a 2-Calabi-Yau category $\Cc$. It is a universal construction, the half Kac–Moody Lie algebra object, on an explicit object $\CG_{\Cc}$. While the construction in \textit{loc. cit.} is carried out in the derived category of mixed Hodge modules, it carries over without any changes to the derived category of $\ell$-adic sheaves. In particular the sheaf $\CG_{\Cc}$ is defined as follows: Let $\Sigma_{\Cc}$ and $\Phi^+_{\Cc} \subset \pi_0(\BM)$ denote the set of primitive positive and simple positive roots of $(\pi_0(\BM), (\cdot,\cdot))$ \cite[Section 3.1]{DHS23}. Then $\CG_{\Cc} = \bigoplus_{\gamma \in \Phi^+_{\Cc}} \CG_{\Cc,\gamma}$ where

\[  \CG_{\Cc,\gamma} = \begin{cases} \IC_{\M_\gamma} \text{ if } \gamma \in \Sigma_{\Cc} \\ (u_\gamma)_*\IC_{\M_{\Cc,\gamma}} \text{ if } \gamma=l\gamma' \text{ with } \gamma' \in \Sigma_{\Cc} \text{ isotropic and } l\geq 2,
\end{cases} \]
and $u_\gamma : \M_{\Cc,\gamma'} \to \M_{\Cc,\gamma}, x \mapsto x^{\oplus l}$. 

The perverse sheaf $\BPS_{\M}$ is then the quotient of the free Lie algebra generated on $\CG_{\Cc}$ by a subsheaf, the Serre ideal. Both are defined with respect to the the symmetric monoidal structure $ \CF_1 \boxdot \CF_2 = \oplus_*(\CF_1 \boxtimes \CF_2)$ on $\mathrm{Perv}(\M)$.

To extend the definition to a 2-Calabi-Yau family $\BM = \BM_{\Cc}$ over $S$ as in Section \ref{2cylin} we consider the relative perverse $t$-structure on $D_c(\M,\BQ_{\ell})$ of \cite{HS23} and copy the definition of $\CG_\Cc$ replacing $\IC_{\M_{\Cc,\gamma}}$ with the intermediate extension in this relative perverse $t$-structure. If $(\pi_0(\BM), (\cdot,\cdot))$ is constant as a group scheme over $S$ we may then directly apply the half Kac–Moody Lie algebra construction to $\CG_\Cc$ to obtain a relative perverse BPS-sheaf $\BPS_{\M}$.

If $(\pi_0(\BM), (\cdot,\cdot))$ becomes constant after a finite étale base change $S' \to S$, the universality of the the half Kac–Moody Lie algebra construction implies that $\BPS_{\M\times_S S'}$ descends to a relative perverse sheaf $\BPS_{\M}$ on $\M$.


\begin{lemma}[Locality principle]\label{lemma:loc-princ}
Let $(X_1,x_1)$, $(X_2,x_2)$ be two adequate moduli spaces $X_1$ and $X_2$ of algebraic stacks $\Xb_1$ and $\Xb_2$ over $k=\mathbb{F}_q$ arising from a 2-Calabi-Yau category. Assume that there exists an equivalence of $k$-stacks $(\Xb_1)_{\widehat{X}_1} \simeq (\Xb_2)_{\widehat{X}_2}$. 
Then, there is an equality
$
\Tr(\varphi_{x_1}|\ {(\BPS_{X_1})_{x_1}}) = \Tr(\varphi_{x_2}|\  {(\BPS_{X_2})_{x_2}}).
$
\end{lemma}
\begin{proof}
 Let us denote by $R_i$ the Henselian local ring attached to $(X_i,x_i)$. The map $R_i \to \widehat{R}_i$ is a regular ring homomorphism and therefore, by Popescu's Theorem, admits a presentation as a colimit of smooth $R_i$-algebras $(A_{i,j})_{j \in J_i}$:
$$\widehat{R}_i = \colim_{j \in J_i} A_{ij}.$$
We remark that each $A_{ij}$ comes with a canonical morphism to $k$, given by the composition $A_{ij} \to \widehat{R}_i \to k$.
 
The assumed equivalence of stacks $(\Xb_1)_{\widehat{X}_1} \simeq (\Xb_2)_{\widehat{X}_2}$ yields an isomorphism of completed rings $\widehat{R}_1 \simeq \widehat{R}_2$, and since the stacks are locally of finite type, this equivalence of stacks can be spread-out to a morphism
$$(\Xb_1)_{A_{1j_1}} \to (\Xb_2)_{A_{2j_2}}$$ 
for appropriately chosen indices $j_i \in J_i$.

Now, we use that $A_{ij_i}$ is a smooth $R_i$-algebra, and that $R_i$ is Henselian, which allows us to lift the canonical homomorphism $A_{ij_i} \to k$ to a homomorphism $A_{ij_i} \to R_i$. Base changing yields morphisms of stacks
$$(\Xb_1)_{R_1}\to \Xb_{A_{2j_2}} \to (\Xb_2)_{R_2}.$$
Proceding in the same manner for the inverse morphism, we obtain
$$(\Xb_2)_{R_2} \to (\Xb_1)_{R_1},$$
and both compositions, are by construction, mutually inverse after passing to $\widehat{R}_i$. Thus, we obtain an equivalence of stacks
\[(\Xb_1)_{R_1}\to (\Xb_2)_{R_2}\]
and an isomorphism of Henselian local rings $R_1 \simeq R_2$.

This isomorphism induces in particular an isomorphism of stalks $(\IC_{X_1})_{x_1} \cong (\IC_{X_2})_{x_2}$. Furthermore by Corollary \ref{indlin}, for any decomposition $x_1 = x_1' \oplus x_1''$ there is a corresponding decomposition $x_2 = x_2' \oplus x_2''$ and the equivalence of $k$-stacks $(\Xb_1)_{\widehat{X}_1} \simeq (\Xb_2)_{\widehat{X}_2}$ induces equivalences of formal completions at $x_1'$, $x_2'$ and $x_1''$, $x_2''$ respectively, compatible with $\oplus$. It follows that we have identifications of stalks of generators $(\CG_{\Cc_1})_{\tilde{x}_1} \cong (\CG_{\Cc_2})_{\tilde{x}_2}$ whenever $\tilde{x}_1$, $\tilde{x}_2$ are corresponding direct summands of $x_1$, $x_2$. Since talking stalks commutes with the half Kac–Moody Lie algebra construction we finally have $(\BPS_{X_1})_{x_1} \cong (\BPS_{X_2})_{x_2}$.\end{proof}

We now consider two specific cases of 2-Calabi-Yau categories which will be relevant in the study of the case of sheaves on a K3 surface.

First, let $Q$ be a $\sigma$-impartial quiver for $\sigma \in \Aut(Q)$ a permutation of the vertices. For any finite field $k \cong \BF_q$ we continue to write $\tau \in H^1(k,\Aut(Q)$ for the class of the cocycle defined by $\sigma \in \Aut(Q)$. We will need the following wall-crossing result. 

\begin{proposition}\label{wc}
Let $\theta \in \Stab_Q$ be a stability condition and  $d \in \BN^{Q_0}$ a dimension vector, both invariant under $\sigma \in \Aut(Q)$. Then there exist constants $p_0$ and $r_0$ such that over every finite field $k \cong \BF_{p^r}$ with $p\geq p_0$ and $r \geq r_0$ we have

\[(p_\theta)_* \BPS_{\N^{\theta,\tau}_{DQ,d} } \cong \BPS_{\N^{\tau}_{DQ,d} },\]
where $p_\theta: \N^{\theta,\tau}_{DQ,d} \to \N^{\tau}_{DQ,d}$ is the natural map to the affine GIT quotient. 
\end{proposition}
\begin{proof}
By the definition of BPS-sheaves it suffices to show that $(p_\theta)_*\IC_{\N^{\theta,\tau}_{DQ,d} } \cong \IC_{\N^{\tau}_{DQ,d}}$. By the decomposition theorem and since $p_\theta$ is birational we know that $i: \IC_{\N^{\tau}_{DQ,d}} \hookrightarrow (p_\theta)_*\IC_{\N^{\theta,\tau}_{DQ,d} }$ is canonically a direct summand. Furthermore, after passing to a finite extension $k'/k$ trivializing the cocylce $\tau$ we have that $(\N^{\theta,\tau}_{DQ,d})_{k'} \cong (\N^{\theta}_{DQ,d})_{k'}$ and $(\N^{\tau}_{DQ,d})_{k'} \cong (\N_{DQ,d})_{k'}$ and thus there is an isomorphism $(p_\theta)_*\IC_{(\N^{\theta,\tau}_{DQ,d})_{k'} } \cong \IC_{(\N^{\tau}_{DQ,d})_{k'}}$ for $k$ sufficiently large implied by the wall crossing formula over $\BC$ \cite[Theorem B]{DM20}, see also \cite[Proposition 7.3.8]{BDNKP25}. Therefore $i$ is already an isomorphism over $k$.    
\end{proof}

The second case of interest are Higgs bundles on a curve: Let as before $R$ be a finite type $\BZ$-algebra $R$ and $S= \Spec(R)$. Given a family of smooth projective curves $\Cc$ over $S$ and integers $r\geq 1, d \in \BZ$ we consider the moduli stack $\BM_{r,d}$ parameterizing semi-stable rank $r$ and degree $d$ Higgs bundles on $\Cc$. We write $\M_{r,d}$ for its adequate moduli space. 

For any rational number $\tau$ we consider the stack
\[ \BM_\tau = \bigsqcup_{d/r = \tau} \BM_{r,d},\]
which defines a CY2 category, see for example \cite[Section 6.3]{DHS22}.
For every $r\geq 1$ we write $\FA_r= \oplus_{i=1}^r H^0(\Cc,K_{\Cc}^i)$ for the Hitchin base and $h_{r,d} \colon \M_{r,d} \to \A_r$ for the Hitchin fibration. Then as a consequence of Kinjo-Koseki's degree independence result for Higgs bundles over $\BC$ \cite{KK21} we deduce similar as in \cite[Proposition 6.4]{GWZ24}

\begin{corollary}\label{cih} For any $r\geq 1$ There exists a non-empty open $S' \subset S$ such that for every finite field $k$, every $\phi \colon \Spec(k) \to S'$ and integers $d,d' \in \BZ$ we have an isomorphism
\[  \left((h_{r,d})_*\BPS_{\M_{r,d}} \right)_\phi \cong  \left((h_{r,d'})_*\BPS_{\M_{r,d'}} \right)_\phi, \]
where the subscript $\phi$ means the restriction to $\A_r \times_S \Spec(k)$ along $\phi$. 
    \end{corollary}

\section{Computing BPS-invariants as a $p$-adic integral}\label{bpsintsect}

\subsection{Recollection on $p$-adic integration}

Let $F$ be a non-archimedean local field, that is a finite extension of $\BQ_p$ of $\BF_p((t))$ for some prime $p$. We denote by $\Oc_F \subset F$ its ring of integers and by $k_F \cong \BF_q$ its residue field. We further fix the Haar measure on $(F,+)$ so that $\Oc_F$ has volume $1$. 

We will freely use the theory of $F$-analytic manifolds, an analogue of the theory of real manifolds, as developed for example in \cite{MR1743467,CNS18}. In particular for an $F$-analytic manifold $M$ of dimension $d$ and an analytic $d$-form $\omega$ on $M$, integration against the absolute value of $\omega$ defines a Borel measure $|\omega|$ on $M$.

For us $F$-analytic manifolds arise naturally from algebraic varieties as follows. Let $X/\Oc_F$ be a reduced, separated finite type scheme of relative dimension $d$ and denote by $X^{sm} \subset X$ the smooth locus of $X$. Then for any non-empty open $U \subset X^{sm}$ the set
\[ X(\Oc_F)^\natural = X(\Oc_F) \cap U(F),  \]
admits the structure of an $F$-analytic manifold, in fact a submanifold of $U(F)$. In particular, any $d$-form $\omega$ on $U$ will induce a measure $|\omega|$ on $X(\Oc_F)^\natural$. 

Reduction modulo the maximal ideal in $\Oc_F$ induces a restriction map $r\colon X(\Oc_F)^\natural \to X(k_F)$ and for any subset $S \subset X(k_F)$ we will often write $X(\Oc_F)^\natural_S$ for $r^{-1}(S)$, the ball of $\Oc_F$-points around $S$.

Finally we will also be interested in computing integrals of functions coming from gerbes: Let $\alpha$ be a $\BG_m$-gerbe on $U$. For any $x\in U(F)$ the class of $x^*\alpha \in H_{\text{\'et}}^2(F,\BG_m)$ can be identified with an element in $\BQ/\BZ$ by means of the isomorphism given by the Hasse invariant $H_{\text{\'et}}^2(F,\BG_m) \cong \BQ/\BZ$. We will write $e^{2\pi i\alpha}$ for the induced map 
\[e^{2\pi i\alpha}\colon X(\Oc_F)^\natural \to \BC \ \ \ \ y \mapsto e^{2\pi i \ (y_{|F})^*\alpha}.\]
Notice that $e^{2\pi i\alpha} \equiv 1$ if $\alpha$ extends to a $\BG_m$-gerbe on all of $X$ since $H_{\text{\'et}}^2(\Oc_F,\BG_m)= 0$.









\subsection{The gerbe-corrected volume computes BPS invariants}\label{sub:vol-BPS}

Let $\BM = \BM_{\Cc}$ over $S$ be a 2-Calabi-Yau family as in Section \ref{2cylin}. Throughout this section we fix a component $\gamma \in \pi_0(\BM)(S)$ such that there exists and open dense $\widetilde{U} \subset \BM_\gamma$, that is a $\BG_m$-gerbe over its image $U \subset \M_\gamma$. 

The 2-Calabi-Yau condition then implies that both $\widetilde{U}$ and $U$ are smooth over $S$ and $U$ admits a non-vanishing, non-degenerate $2$-form $\theta_{U}$. We write $\omega_{U}$ for the induced non-vanishing top-form on $U$. 

For any morphism $\Spec(\Oc_F) \to S$ from the ring of integers of a non-archimedean local field $F$ we simply write $\M_\gamma(\Oo_F)$ for $\M_\gamma \times_S \Spec(\Oc_F)$ and similarly $\M_\gamma(F)$, $U(F)$ etc. In particular we consider the $F$-analytic manifold 
\[\M_\gamma(\Oc_F)^\natural = \M_\gamma(\Oc_F) \cap U(F), \]
together with the measure $|\omega_{\M_\gamma}| = |\omega_U|$ induced by the symplectic form on $U$. Finally we will consider the $\BG_m$-gerbe $\alpha \colon\widetilde{U}\to U$ and its induced function $e^{2\pi i \alpha}$ on $\M_\gamma(\Oc_F)^\natural$. Since the complement of $\M_\gamma(\Oc_F)^\natural$ in $\M_\gamma(\Oc_F)$ has measure $0$ we will drop the superscript $\natural$ for the rest of the article as it will not affect the integrals we consider.

As an immediate corollary of Proposition \ref{prop:K3-linear} we have the following locality principle for gerbe-corrected volumes.

\begin{corollary}\label{ploc} Let $S' \subset S$ be as in Proposition \ref{prop:K3-linear}. Then for any $\Spec(\Oc_F) \to S'$ and any $x \in \M_\gamma(k_F)$ we have
\[\int_{\M_\gamma(\Oc_F)_{x}} e^{2\pi i \alpha} |\omega_{\M_\gamma}| = \int_{\N_{DQ_x,d_x}^{\tau_x}(\Oc_F)_{0}} e^{2\pi i \alpha} |\omega_{\N_{DQ_x,d_x}^{\tau_x}}|,\]
with $\N_{DQ_x,d_x}^{\tau_x}$ the twisted quiver at $x$ as in Proposition \ref{prop:K3-linear}.
\end{corollary}
\begin{proof} Notice that $\M_\gamma(\Oc_F)_{x}$ and the gerbe function $e^{2\pi i \alpha}$ only depend on the completion $(\widehat{\BM}_\gamma)_x \to (\widehat{\M}_\gamma)_x$. Since the equivalence of Proposition \ref{prop:K3-linear} also preserves symplectic forms the corollary follows.     
\end{proof}


We are now ready to state the main result of this section. 

\begin{theorem}\label{thm:key}
There exists an open subscheme $S_\gamma \subset S$  such that for every morphism $\Spec(\Oc_F) \to S$ from the ring of integers of a local field $F$ with residue field $k_F$ of cardinality $q$, and for every $x \in \M_\gamma(k_F)$ we have an equality
\[
\int_{\M_\gamma(\Oc_F)_{x}} e^{2\pi i \alpha} |\omega_{\M_\gamma}| = \frac{\Tr(\varphi_{x}|(\BPS_{\M_\gamma})_{x})}{q^{\dim \M_\gamma}}.
\]
\end{theorem}

We will prove this at the end of this subsection. At first we need to establish the equality for linear models, that is, for twisted quiver varieties. Evaluating gerbe-corrected integrals directly, even in these explicit cases, turns out to be rather challenging. The proof of the quiver case, crucial to the overall strategy, hinges on $\chi$-independence for Higgs moduli \cite{KK21}. Indeed, the proof recursively computes the local gerbe-corrected volumes for a given point in a twist of a quiver variety for $Q$, in a procedure that replaces the pair $(Q,d)$ by $(Q',d')$ with the new quiver $Q'$ being a contraction of the first, and the dimension vector $d'$ having strictly smaller mass: $|d'| < |d|$. This recursion grinds to a halt, when there is no non-trivial stability condition on the twisted quiver variety. This happens if and only if the Galois-action acts transitively on the set of vertices. Fortunately, this type of singularity is always achieved by a suitably chosen Hitchin system. Using that $\chi$-independence holds for Higgs moduli, we may conclude that gerbe-corrected volumes compute traces of Frobenius on $\BPS$-sheaves in these remaining cases.


\begin{lemma}[Key Lemma]\label{lemma:key}
Let $(DQ_{R'},\tau)$ be a $\sigma$-impartial twisted quiver enriched in $R'$-modules. Let $\theta \in \Stab_{DQ}$ a stability condition and  $d \in \BN^{Q_0}$ a dimension vector, both invariant under $\sigma \in \Aut(Q)$. Assume that the locus of simple representations is dense in $\N_{DQ,d}^\tau$. Then there exist constants $p_0$ and $r_0$ such that for every finite field $k \cong \BF_{p^r}$ with $p\geq p_0$ and $r \geq r_0$, any local field $F$ with residue field $k$ and any $x \in \N_{DQ,d}^\tau(k)$ we have

\[
\int_{\N_{DQ,d}^\tau(\Oc_F)_{x}} e^{2\pi i \inv(\alpha)} |\omega_{\N_{DQ,d}^\tau}| = \frac{\Tr(\varphi_{x}|(\BPS_{\N_{DQ,d}^\tau})_{x})}{q^{\dim \N_{DQ,d}^{\tau}}}.
\]
\end{lemma}
\begin{proof}

First, note that if $x$ lies in the smooth locus of $\N_{DQ,d}^\tau$, then both sides of the equality equal $q^{-\dim(\N_{DQ,d}^\tau)}$. Thus we assume from now on that $x$ lies in the singular locus, in particular, it represents a decomposable object. Then by Proposition \ref{prop:K3-linear} we may further assume that $x=0\in \N_{DQ,d}^\tau$ represents the trivial representation. 
We will now prove the result by induction. Let $n$ be an integer and assume that the equality is known for all pairs $(Q,d)$ consisting of a quiver and a dimension vector $d$ such that $|Q_0| \leq n$ and $|d| \leq n$. Note that the assertion is vacuously true for $n=0$. We will show that the equality also holds for $n+1$. 
There are two cases to consider.

\medskip

\underline{First case:} $\dim(\Stab_Q)^\sigma \geq 1$. In this case, let $\theta$ denote a generic stability condition for the quiver $Q$. The resulting map $c\colon \N^{\theta,\tau}_{DQ,d} \to \N^{\tau}_{DQ,d}$ is a partial volume-preserving resolution, as it preserves the symplectic forms on the smooth loci. 

Any point $z \in \BN^{\theta^{\rm gen},\tau}_{DQ,d}(k)$ lying above $x$ represents an object with less indecomposable summands than $x$. Twisted quiver varieties can be symplectically linearised since the satisfy the assumptions of Corollary \ref{cor:symp-lin} with respect to an \'etale cover consisting of a single \'etale morphism obtained by base change along $\Spec k' \to \Spec k$.  Linearising at $z$ thus yields a new quiver $DQ_{z}$ with $|DQ_{z,0}| < |DQ_0|$ and $|d_z| < |d|$, endowed with a Galois cocycle encoding a twist. The property of being $\sigma$-impartial is inherited to the descendants $(DQ_z,\tau)$; the proof is included below (see Lemma \ref{lemma:heritative}). 
We then obtain
\[
\int_{\N_{DQ,d}^\tau(\Oc_F)_{x}} e^{2\pi i \inv(\alpha)} |\omega_{\N_{DQ,d}^\tau}|  = \int_{\N^{\theta,\tau}_{DQ,d}(\Oc_F)_{c^{-1}(x)}} e^{2\pi i \inv(\alpha)}|\omega_{\N_{DQ,d}^{\theta,\tau}}| =  \sum_{z\in c^{-1}(x)}\int_{\N^{\tau}_{DQ_z,d_z}(\Oc_F)_{z}} e^{2\pi i \inv(\alpha)}.
\]
By virtue of our inductive hypothesis, the right-hand side is equal to 
\[
\sum_{z\in c^{-1}(x)}\frac{\Tr(\varphi_{z}|(\BPS_{\N^{\tau}_{DQ_z,d_z}})_{z})}{q^{\dim \N^\tau_{Q_z,d_z}}} =  \frac{\Tr(\varphi_{x}|(\BPS_{\N_{DQ,d}^\tau})_{x})}{q^{\dim \N_{DQ,d}^\tau}},
\]
where we have used $c_*\BPS_{\N_{DQ,d}^{\theta,\tau}} = \BPS_{\N_{DQ,d}^\tau}$ by Proposition \ref{wc} and the principle of locality \ref{lemma:loc-princ}.

\medskip

\underline{Second case:} $\dim(\Stab_Q)^\sigma =0$. In this case the vertices of $Q$ must form a single $\sigma$-orbit. We will use a global argument using $\chi$-independence for Higgs bundles \cite{KK21}. Namely there exists a smooth projective curve $C$ over $k$ of genus $g\geq 2$, an integer $r \geq 1$ and a section of the Hitchin base $a \in \FA_r= \oplus_{i=1}^n H^0(C,K_C^i)$ such that $(\N_{DQ,d}^\tau,0)$  is the local model for the most singular point $y$ in the Hitchin fiber $h_{r,0}^{-1}(a)\subset \M_{r,0}$, see Construction \ref{const:spectral-curve} below, where we define $a=a_{n,\theta}$. 

Choosing $p_0$ and $r_0$ big enough we may further assume that there exists a finite type $\BZ$-algebra $R$, a family of smooth projective curves $\Cc$ over $S=\Spec(R)$ and a $S$-section $\tilde a \in \A_r(S)$ of the Hitchin base such that the above picture over $\Spec(k)$ arises as the pullback along a morphism $\phi\colon \Spec(k) \to S$. 
Now let $\tilde{\phi} \colon \Spec(\Oc_F) \to S$ be any homomorphism restricting to $\phi$. By means of the locality principles \ref{lemma:loc-princ} and \ref{ploc} it is therefore enough to show
\begin{equation}\label{higgsbps}
\int_{\M_{r,0}(\Oc_F)_{y}} e^{2\pi i \inv(\alpha)} |\omega_{\M_{r,0}}| = \frac{\Tr(\varphi_{y}|(\BPS_{\M_{r,0})_{y})}}{q^{\dim \M_{r,0}}}.
\end{equation}

Again, by the locality principles and the induction hypothesis \eqref{higgsbps} holds if we replace $y$ by any $y\neq y'\in h^{-1}_{r,0}(a)(k)$. Finally, we also know that integrating along the Hitchin fibre we have
\[ 
\int_{\M_{r,0}(\Oc_F)_{h_{r,0}^{-1}(a)(k)}} e^{2\pi i \inv(\alpha)} |\omega_{\M_{r,0}}| = \frac{\Tr(\varphi_{a}|((h_{r,0})_*\BPS_{\M_{r,0}})_{a})}{q^{\dim \M_{r,0}}},
\]
since both sides are remain unchanged when replacing $\M_{r,0}$ with any $\M_{r,d}$ with $d$ comprime to $r$. For the right-hand side this follows from Corollary \ref{cih} while for the left hand side this follows from the main result of \cite[Theorem 5.0.2 \& Remark 5.0.9]{COW21}. Combining these two observations we deduce \eqref{higgsbps}. 
\end{proof}

We now provide the two remaining proofs of the assertions used above.

\begin{lemma}[$\sigma$-impartiality for descendants]\label{lemma:heritative}
The property of $\sigma$-impartiality is inherited to the descendant quivers $(DQ_z,\tau_z)$ constructed by means of linearisation of a partial resolution $\BN_{DQ}^{\tau,\theta} \to \BN_{DQ}^{\tau}$.     
\end{lemma}
\begin{proof}
Since $DQ_z$ is a doubled quiver, $\sigma$-impartiality is equivalent to the assertion
$$(x_i,x_j) = (x_i,\sigma(x_j)) \text{ }\forall i,j \in (DQ_z)_0.$$
This property holds since the resolution $\BN_{DQ}^{\tau,\theta} \to \BN_{DQ}^{\tau}$ is induced by the inclusion of the category of $\theta$-semistable representations $\iota\colon \Rep^{\theta-ss,\tau}(\Pi_Q) \to \Rep^{\tau}(\Pi_Q)$ in the category of all representations. This functor induces an isomorphism on the Grothendieck groups $K_0$ and thus preserves the bilinear symmetric form $(-,-)$, which is defined by $\chi(RHom(-,-))$, and naturally factors through the $K$-groups. Thus, for every pair of objects $x_i,x_j \in \Rep^{\theta,\tau}$ we have the following chain of equalities
$$(x_i,x_j) = (\iota(x_i),\iota(x_j)) = (\iota(x_i),\sigma(\iota(x_j))) = (x_i,\sigma(x_j)).$$
This concludes the proof that $\sigma$-impartiality is inherited to the descendant quivers.
\end{proof}

\begin{construction}\label{const:spectral-curve}
Let $C/\Oc_F$ be a smooth curve of genus $g$. The $\Oc_F$-module of $1$-forms $H^0(C,\Omega^1_{C/F})$ is then of rank $g$. We consider an unramified extension $L/F$ of degree $n$ and denote the Frobenius element generating the cyclic group $\Gal(L/F)$ by $\sigma$. We then choose an element $\theta \in H^0(C_L,\Omega^1_{C_L/L}) = H^0(C,\Omega^1_{C/F}) \otimes_F L$ such that $\sigma^{\ell}(\theta) \neq \theta$ whenever $\ell \leq n-1$. Note that we automatically have that $\sigma^n(\theta) = \theta$ since $\theta$ is defined over $L/F$. 

For any $L \in \Pic^0(C)(\Oc_L)$, we may then consider the rank $n$ Higgs bundle
$$\oplus_{\ell=0}^{n-1}(L,\theta)^{\sigma^{\ell}},$$
which has characteristic polynomial
$$a_{n,\theta}=\prod_{\ell=0}^{n}(\lambda-\sigma^{\ell}(\theta))$$
by construction. Since the Higgs bundle defined above splits by definition into simple factors acted on transitively by the Galois group, we obtain an associated quiver that is $\sigma$-impartial (see Lemma \ref{lemma:sigma-impartial}) and with $\langle \sigma \rangle$ acting transitvely on the set of vertices.
\end{construction}

With the gerbe-corrected volume being understood for local models, we may turn to the proof of the general case.
\begin{proof}[Proof of Theorem \ref{thm:key}]
We choose a spreading-out of this diagram of pointed spaces, which is $\gamma$-adapted to a ring $R \subset \Cb$ which is finitely generated over $\Zb$. 
It follows from Lemma \ref{prop:K3-linear} that there exists a quiver $Q_{x}$ with twist $\tau$, such that there is an equivalence of stacks
$$
(\Mb_\gamma)_{(\widehat{M_\gamma})_x}  \simeq (\BN_{Q_{x}}^{\tau})_{(\widehat{\N_{DQ_{x}}^{\tau}})_0}.
$$
Corollary \ref{ploc} and Lemma \ref{lemma:key} then imply
\[
\int_{\M_\gamma(\Oc_F)_{x}} e^{2\pi i \alpha}|\omega_\gamma| =\int_{\N_{DQ_{x}}^{\tau}(\Oc_F)_{0}} e^{2\pi i \alpha} |\omega_{\N_{DQ_{x}}^{\tau}}| = \frac{\Tr(\varphi_{0}|(\BPS_{\N_{DQ_{x}}^{\tau}})_{0})}{q^{\dim \N_{DQ_x}}}.
\]
Using the \emph{locality principle} for BPS trace-functions, Lemma \ref{lemma:loc-princ}, we conclude that the right-hand side is equal to 
\[
\frac{\Tr(\varphi_{x}|(\BPS_{\M_\gamma})_{x})}{q^{\dim \M_\gamma}}
\]
as asserted.
\end{proof}

\begin{rmk} Theorem \ref{thm:key} has also some interesting consequences for the computation of $p$-adic integrals. Via the cohomological integrality identity one obtains from \cite[Theorem 5.12]{GWZ24} with the notations of \textit{loc. cit.} the identity

\[ \int_{\M_\gamma(\Oc_F)} e^{2\pi i \alpha} |\omega_{\M_\gamma}| = - \lim_{T\to \infty} \sum_{r \geq 1} \#^{\Bw,\alpha} I_{\mu_r}\Mc_{\gamma} T^r.  \]
This strongly resembles the integration formula for smooth algebraic stacks \cite[Theorem 4.9]{GWZ24} despite the presence of singularities. Our current proof only works when integrating $e^{2\pi i \alpha}$; it would be desirable for the formula to hold for any admissible function.
\end{rmk}

\section{The proof of $\chi$-independence}\label{proofsect}

\subsection{Fibrewise integration}

Let $X_{\BC}$ be a complex smooth projective surface with trivial canonical bundle, that is either a K3 or abelian surface. Let $\beta_\BC$ be an ample, base-point free curve class on $X_\BC$ and $\chi \in \BZ$. We write $\BM_{\beta,\chi,\BC}$ for the moduli stack of pure $1$-dimensional Gieseker-semistable sheaves on $X_{\BC}$ with support of class $\beta_\BC$ and Euler-characteristic $\chi$. Here, semistability is with respect to a fixed polarization $L_\BC$ and the slope function
\[\mu(\CF) = \frac{\chi(\CF)}{c_1(\CF) \cdot L_\BC}.\]
For $\tau \in \BQ$, we also write $\BM_{\tau,\BC}$ for the disjoint union of the moduli stacks $\BM_{\beta,\chi,\BC}$ with fixed slope $\tau$, which forms a 2-Calabi-Yau category \cite[Section 1.1.1]{Da24}. 

In order to state $\chi$-independence consider the Hilbert-Chow morphism 
\[h_{\beta,\chi}: \M_{\beta,\chi,\BC} \to B_\BC=\mathbb{P} H^0(X,\CO_X(\beta)),\] 
sending a $1$-dimensional sheaf to its fitting support.

\begin{theorem}\label{chiind} For any $\chi,\chi' \in \BZ$ there is an isomorphism 
\[(h_{\beta,\chi})_*\BPS_{\M_{\beta,\chi,\BC}} \cong (h_{\beta,\chi'})_*\BPS_{\M_{\beta,\chi',\BC}} \]
in $D_c(B_\BC,\BQ_\ell)$. Furthermore, when considered as mixed Hodge modules, the $E$-polynomials of the stalks $\left((h_{\beta,\chi})_*\BPS_{\M_{\beta,\chi,\BC}}\right)_b$ and $\left((h_{\beta,\chi'})_*\BPS_{\M_{\beta,\chi',\BC}})\right)_b$ agree for every $b \in B_\BC$. 
\end{theorem}

\begin{proof}
We will prove Theorem \ref{chiind} by passage to finite fields. For this we choose a spreading out $X$ of $X_\BC$ over $S=\Spec(R)$ for a finitely generated $\BZ$-algebra $R$, together with its polarization $L$ and curve class $\beta$ to obtain a family $\BM_{\tau}/S$ and GIT-quotients $\tilde{\pi}_{\beta,\chi}: \BM_{\beta,\chi} \to \M_{\beta,\chi}$ satisfying conditions $(i)$ and $(ii)$ above, see for example \cite[Section 2.1]{COW21}. By the same argument as in \cite[Theorem 6.6]{GWZ24} it suffices to show that there is a non-empty open $S' \subset S$ such that for every finite field $k$, any $\phi: \Spec(k) \to S'$ and any $b \in B_\phi(k)$ we have an equality of traces

\[ \Tr(\varphi_b,(h_{\beta,\chi*}\BPS_{\M_{\beta,\chi},\phi})_b )= \Tr(\varphi_b,(h_{\beta,\chi'*}\BPS_{\M_{\beta,\chi'},\phi})_b ).  \]
Picking $S'=S_\gamma$ as in Theorem \ref{thm:key} we may reformulate the equality of traces as the equality of integrals 

\[ \int_{\M_{\beta,\chi}(\Oc_F)_{h_{\beta,\chi}^{-1}(b)}} e^{2\pi i{\alpha}} \ |\omega_{\M_{\beta,\chi}}| = \int_{\M_{\beta,\chi'}(\Oc_F)_{h_{\beta,\chi'}^{-1}(b)}} e^{2\pi i{\alpha}} \ |\omega_{\M_{\beta,\chi'}}|.  \]
For K3 surfaces this is the content of \cite[Theorem 1.2.2]{COW21}. The proof for abelian surfaces is identical, but we recall here the main steps of the argument, which goes back to \cite{GWZ20a}. Let $B^{sm} \subset B$ denote the dense open locus of smooth curves inside the linear system. Then by  Fubini it is enough to show for any $\tilde{b} \in B_\phi(\Oc_F)_b \cap B_\phi^{sm}(F)$ the equality
\begin{equation}\label{fineq}    
 \int_{h_{\beta,\chi}^{-1}(\tilde b)(\Oc_F)} e^{2\pi i{\alpha}} \ |\omega_b| =  \int_{h_{\beta,\chi'}^{-1}(\tilde b)(\Oc_F)} e^{2\pi i{\alpha}} \ |\omega_b|.\end{equation} 
 To explain the notation $|\omega_b|$ write $C_{\tilde{b}} \subset X_{\Oc_F}$ for the generically smooth curve parametrized by $\tilde b$ and $g=g(C_{\tilde b})$ for its genus. We then have an identification over $\Spec(F)$  of $h_{\beta,\chi}^{-1}(\tilde b)$ with the Picard-variety $\Pic^{\chi +g-1}_{C_{\tilde{b}}}$
and similarly for $\chi'$. Then, as shown in Subsection \ref{VolumeFormSubsec} below, $|\omega_b|$ is induced on both sides of \eqref{fineq} from the same global translation-invariant form on $\Pic^{0}_{C_{\tilde{b}}}$ \cite[Lemma 6.13]{GWZ20a}. Furthermore the gerbe $\alpha: \BM_{\beta,\chi}^{st}\to\M_{\beta,\chi}$ is identified with $\alpha_{\chi+g-1}\colon \BPic_{C_{\tilde b}}^{\chi+g-1}\to \Pic_{C_{\tilde b}}^{\chi+g-1}$, where $\BPic_{C_{\tilde b}}^{\chi+g-1}$ denotes the Picard-stack. 

To prove \eqref{fineq} let us assume for simplicity that $\chi'=1$. Then semi-stability and stability agree and thus $\BM_{\beta,\chi}\to\M_{\beta,\chi}$ is a $B\BG_m$-gerbe. Since $H^2(\Spec(\Oc_F),B\BG_m)=0$ we then have $e^{2\pi i{\alpha}} \equiv 1$ and \eqref{fineq} becomes
\begin{equation} \label{cbc}   
 \int_{\Pic^{\chi +g-1}_{C_{\tilde{b}}}(F)} e^{2\pi i{\alpha_{\chi+g-1}}} \ |\omega_b| =  \int_{\Pic^{g}_{C_{\tilde{b}}}(F)} |\omega_b|.\end{equation} 

Now let us assume first that $\Pic^{g}_{C_{\tilde{b}}}(F) = \emptyset$. In that case we need to see that also the left hand side of \eqref{cbc} equals $0$. If also $\Pic^{\chi +g-1}_{C_{\tilde{b}}}(F)= \emptyset$ this is automatic. Otherwise any rational point induces an isomorphism $\Pic^{\chi +g-1}_{C_{\tilde{b}}}\cong \Pic^{0}_{C_{\tilde{b}}}$ and under this identification we have $e^{2\pi i{\alpha_{\chi+g-1}}} = c \cdot e^{2\pi i{\alpha_{0}}} $ for some non-zero constant $c$ \cite[(4.3.1)]{COW21}. Finally, by \cite[Lemma 5.0.6]{COW21} since $\Pic^{g}_{C_{\tilde{b}}}(F) = \emptyset$ we get that $e^{2\pi i{\alpha_{0}}}$ defines a non-trivial character on $\Pic^{0}_{C_{\tilde{b}}}(F)$ and hence also the left hand side of \eqref{cbc} is $0$. 

Similarly if $\Pic^{g}_{C_{\tilde{b}}}(F) \neq \emptyset$, then also $\Pic^{1}_{C_{\tilde{b}}}(F) \neq \emptyset$ by \cite[Lemma 5.0.6]{COW21} and thus $\Pic^{d}_{C_{\tilde{b}}}(F) \neq \emptyset$ for any $d \in \BZ$. It thus remains to show that $\alpha_{\chi+g-1} \equiv 0$. Again by \cite[Lemma 5.0.6]{COW21} we know that the order of $\alpha_{\chi+g-1}$ divides $g$, where $g=\frac{1}{2}\beta^2+1$ since $S$ has trivial canonical bundle. On the other hand, if we write $(\beta,\chi) = l(\beta_0,\chi_0)$ with $l\geq 1$ and $(\beta_0,\chi_0)$ primitive, then it follows from \cite[Lemma 5.18]{GWZ24} that the order of $\alpha_{\chi+g-1}$ also divides $l$. Since $g$ and $l$ are coprime we deduce $\alpha_{\chi+g-1} \equiv 0$ which finishes the proof.  
\end{proof}

\subsection{Comparison of volume forms} \label{VolumeFormSubsec}
Let $\tilde b \in B_\phi(\Oc_F)_b \cap B_\phi^{sm}(F)$ as above. We show that the volume form $\omega_b$ on $\Pic^{0}_{C_{\tilde{b}}}$ induced by the symplectic form $\theta$ on $\M_{\beta,\chi,b}$ is independent of $\chi$ by adapting the corresponding argument from \cite{SHEN2024109616}.

In the following we only work over the stable locus of $\Mb$. We consider the Hecke scheme $\Hs$ classifying pairs $(E \subset F, x)$ consisting of an inclusion of sections $E,F \in \M^{\St}(S)$ and a section $x\colon S \to X$ for which the quotient $F/E$ is equal to the skyscraper sheaf supported at $x$. The proof of the existence of this scheme and of its smoothness given in \cite[Section 2]{PMIHES_2022__135__337_0} over the complex numbers directly carries over to the current situation. The scheme $\Hs$ comes with projections $p_1,p_2 \colon \Hs \to \M^{\St}$ as well as $p_X\colon \Hs \to X$.

\begin{lemma} \label{ThetaPullbacks}
    $p_1^*\theta=p_2^*\theta$
\end{lemma}
\begin{proof}
 We denote by $E$ the universal coherent sheaf over $\M^{\St}$. The tangent bundle of $\M^{\St}$ is given by $\underline{Ext}^1(E,E)$ and the symplectic form $\theta$ is given by the pairing
 \begin{equation*}
     \underline{Ext}^1(E,E) \times \underline{Ext}^1(E,E) \to \underline{Ext}^2(E,E)\cong \underline{Hom}(E,E) \cong \CO_{\M}
 \end{equation*}
 given by cup product and Serre duality.

 Similarly, as in \cite[6.11]{PMIHES_2022__135__337_0}, the tangent bundle of $\Hs$ is given by the kernel
 \begin{equation*}
    \ker(\underline{Ext}^1(p_1^*E,p_1^*E) \oplus \underline{Ext}^1(p_2^*E,p_2^*E) \to \underline{Ext}^1(p_1^*E,p_2^*E))
 \end{equation*}
 given by the difference of the two morphisms induces by the universal inclusion $p_1^*E \into p_2^*E$ over $\Hs$. Here the differentials of the projections $p_i$ are given by the two projections. Taking factor-wise cup product defines a pairing on this kernel with target
 \begin{equation*}
     \ker(\underline{Ext}^2(p_1^*E,p_1^*E) \oplus \underline{Ext}^2(p_2^*E,p_2^*E) \to \underline{Ext}^2(p_1^*E,p_2^*E)) \cong \CO_{\Hs}.
 \end{equation*}

 One sees that this pairing agrees with both pullbacks of $\theta$ to $\Hs$.
\end{proof}

Consider $\tilde b$ as above and the resulting smooth curve  $C_{\tilde b} \subset X_F$. Since it suffices to prove the desired independence of $\omega_{\tilde b}$ from $\chi$ after replacing $F$ by a finite extension, we may assume that there exists a point $c \in C_{\tilde b}(F)$. We fix such a point. Since the two sheaves $p_i^*E$ on $\Hs$ differ only by a modification, the two compositions $\Hs \to \M \to B$ agree. We let $\Hs_b \subset \Hs$ be the fiber over $b$ of this map, and $\Hs_{b,c} \subset \Hs_b$ the fiber over $c$ under $p_X$. Then we can consider the diagram
\begin{equation*}
    \xymatrix{
& \Hs_{{\tilde b},c} \ar[d] \ar[ldd]_{\cong} \ar[rdd]^{\cong} & \\
& \Hs_{{\tilde b}}  \ar[ld]^{p_1} \ar[rd]_{p_2} & \\
\M_{\tilde b} & & \M_{\tilde b}    
    }
\end{equation*}
in which the two isomorphisms realize the isomorphism $\M_{\tilde b} \to \M_{\tilde b}$ of $\Pic^{0}_{C_{\tilde{b}}}$-torsors given by modification at $c$ which shifts the degree by one. Using Lemma \ref{ThetaPullbacks} this gives the desired independence of $\omega_b$ from the degree.
\bibliographystyle{amsalpha}
\bibliography{master}
\end{document}